\documentclass[11pt]{amsart}
\usepackage[margin=2.75cm]{geometry}

\usepackage[noend,boxruled]{algorithm2e}
\SetKwFor{For}{\normalfont{for}}{:}{endfor}

\usepackage[english]{babel}
\usepackage{appendix}
\usepackage{amsmath}
\usepackage{amsfonts}
\usepackage{amssymb}
\usepackage{amsthm}
\usepackage{marginnote}
\usepackage{stmaryrd}
\usepackage{enumitem}
\usepackage[english]{babel}
\usepackage{yfonts}
\usepackage[T1]{fontenc}
\usepackage[utf8x]{inputenc}
\usepackage[normalem]{ulem}
\usepackage[numbers]{natbib}

\usepackage[backref=page]{hyperref}
\hypersetup{
  colorlinks   = true,          %Colors links instead of ugly boxes
  urlcolor     = red,          %Color for external hyperlinks
  linkcolor    = blue,          %Color of internal links
  citecolor   = violet             %Color of citations
}

\usepackage{calrsfs}
\DeclareMathAlphabet{\pazocal}{OMS}{zplm}{m}{n}

  \makeatletter
\def\@tocline#1#2#3#4#5#6#7{\relax
  \ifnum #1>\c@tocdepth % then omit
  \else
    \par \addpenalty\@secpenalty\addvspace{#2}%
    \begingroup \hyphenpenalty\@M
    \@ifempty{#4}{%
      \@tempdima\csname r@tocindent\number#1\endcsname\relax
    }{%
      \@tempdima#4\relax
    }%
    \parindent\z@ \leftskip#3\relax \advance\leftskip\@tempdima\relax
    \rightskip\@pnumwidth plus4em \parfillskip-\@pnumwidth
    #5\leavevmode\hskip-\@tempdima
      \ifcase #1
       \or\or \hskip 1em \or \hskip 2em \else \hskip 3em \fi%
      #6\nobreak\relax
    \dotfill\hbox to\@pnumwidth{\@tocpagenum{#7}}\par
    \nobreak
    \endgroup
  \fi}
\makeatother

\usepackage{verbatim}
\usepackage{graphicx}
\usepackage{verbatim}
\usepackage{faktor}
\usepackage{xcolor}
\usepackage{xfrac}
\usepackage{tikz,tikz-cd}
\usetikzlibrary{decorations.pathmorphing,decorations.pathreplacing,patterns,arrows.meta,shapes.misc}

\definecolor{red}{rgb}{1,0,0}

\usepackage[all]{xy}
\usepackage{bbm}
\usepackage{tabularx}
\usepackage{longtable}
\usepackage{tabu}
\usepackage{booktabs}
\usepackage{mathtools}

\usepackage[]{textcomp}
\usepackage[sups]{Baskervaldx}
\usepackage{cabin}
\usepackage[varqu,varl]{inconsolata}
\usepackage[baskervaldx,bigdelims,vvarbb]{newtxmath}
\usepackage[cal=cm]{mathalfa}

\DeclareRobustCommand{\plE}{\text{\reflectbox{$\exists$}}}
\newcommand{\plC}{\scalebox{0.8}[1.3]{$\sqsubset$}}

\newcommand{\lcm}{\operatorname{lcm}}

\def\Yagraph{\tikz[baseline=-3pt,scale=.8]{
\draw (2,0) -- (0,1) (2,0) -- (0,.5) (2,0) -- (0,-1);
\draw (2,0) circle(2pt)[fill=black];
\draw [->] (2,0) -- (2.7,0);
\draw (2.6,0) node[right]{\tiny{$x_k$}};
\draw (0,1) circle(2pt)[fill=white];
\draw (0,.5) circle(2pt)[fill=black];
\draw (0,-.25) node{$\vdots$};
\draw (0,-1) circle(2pt)[fill=black];
\draw [blue,fill=blue] (0,-2) circle[radius=2pt];
\draw [->,blue,thick] (0,-2) -- (3,-2);
\draw [blue,fill=blue] (2,-2) circle[radius=2pt];
\draw [->,blue] (1,-0.9) -- (1,-1.6);

\draw [blue] (2,-2) node[above]{\tiny$\diamond$};
\draw [blue] (0,-2) node[above]{\tiny$0$};

\draw (2.05,0) node[above]{\tiny{$\sqC_0$}};
\draw (0,1) node[left]{\tiny{$\sqC_1$}};
\draw (0,.5) node[left]{\tiny{$\sqC_2$}};
\draw (0,-1) node[left]{\tiny{$\sqC_r$}};
}}

\def\Ybgraph{\tikz[baseline=-3pt,scale=.8]{
\draw (2,0) to[out=120,in=0] (0,1) (2,0) -- (0,1) (2,0) -- (0,.5) (2,0) -- (0,-1);
\draw (2,0) circle(2pt)[fill=black];
\draw [->] (2,0) -- (2.7,0);
\draw (2.6,0) node[right]{\tiny{$x_k$}};
\draw (0,1) circle(2pt)[fill=black];
\draw (0,.5) circle(2pt)[fill=black];
\draw (0,-.25) node{$\vdots$};
\draw (0,-1) circle(2pt)[fill=black];
\draw [blue,fill=blue] (0,-2) circle[radius=2pt];
\draw [->,blue,thick] (0,-2) -- (3,-2);
\draw [blue,fill=blue] (2,-2) circle[radius=2pt];
\draw [->,blue] (1,-0.9) -- (1,-1.6);
\draw [blue] (2,-2) node[above]{\tiny$\diamond$};
\draw [blue] (0,-2) node[above]{\tiny$0$};

\draw (2.05,0) node[above]{\tiny{$\sqC_0$}};
\draw (0,1) node[left]{\tiny{$\sqC_1$}};
\draw (0,.5) node[left]{\tiny{$\sqC_3$}};
\draw (0,-1) node[left]{\tiny{$\sqC_r$}};
}}

\def\Ycgraph{\tikz[baseline=-3pt,scale=.8]{
\draw (2,0) -- (0,1) (2,0) -- (0,.5) (2,0) -- (0,-1);
\draw [->] (2,0) -- (2.7,0);
\draw (2.6,0) node[right]{\tiny{$x_k$}};
\draw (2,0) circle(2pt)[fill=white];
\draw (0,1) circle(2pt)[fill=black];
\draw (0,.5) circle(2pt)[fill=black];
\draw (0,-.25) node{$\vdots$};
\draw (0,-1) circle(2pt)[fill=black];
\draw [blue,fill=blue] (0,-2) circle[radius=2pt];
\draw [->,blue,thick] (0,-2) -- (3,-2);
\draw [blue,fill=blue] (2,-2) circle[radius=2pt];
\draw [blue] (2,-2) node[above]{\tiny$\diamond$};
\draw [blue] (0,-2) node[above]{\tiny$0$};
\draw [->,blue] (1,-0.9) -- (1,-1.6);

\draw (2.05,0) node[above]{\tiny{$\sqC_0$}};
\draw (0,1) node[left]{\tiny{$\sqC_1$}};
\draw (0,.5) node[left]{\tiny{$\sqC_2$}};
\draw (0,-1) node[left]{\tiny{$\sqC_r$}};
}}
\newcommand{\sqC}{\scalebox{0.8}[1.3]{$\sqsubset$}}
\newcommand{\Kim}{\operatorname{Kim}}

\newcommand{\M}[4]{\overline{\mathcal{M}}_{#1,#2}(#3,#4)}

\newcommand{\MLog}{\overline{\mathcal{M}}^{\operatorname{log}}}

\newcommand{\PP}{\mathbb P}
\newcommand{\Z}{\mathbb{Z}}
\newcommand{\VZ}{\pazocal{V\!Z}}

\newcommand{\st}{\star}

\newcommand{\N}{\mathbb{N}}
\newcommand{\OO}{\mathcal{O}}
\renewcommand{\to}{\rightarrow}

\newcommand{\Aaff}{\mathbb{A}}

\newcommand{\Gm}{\mathbb{G}_{\text{m}}}
\newcommand{\virt}[1]{[#1]^{\operatorname{virt}}}

\newcommand{\CC}{\mathbb{C}}

\newcommand{\Glog}{\mathbb{G}_{\mathrm{log}}}

\newcommand{\bq}{\begin{equation}}
\newcommand{\eq}{\end{equation}}
\newcommand{\ba}{\begin{aligned}}
\newcommand{\ea}{\end{aligned}}
\newcommand{\be}{\begin{enumerate}}
\newcommand{\ee}{\end{enumerate}}
\newcommand{\bsm}{\left(\begin{smallmatrix}}
\newcommand{\esm}{\end{smallmatrix}\right)}                   
\newcommand{\bpm}{\begin{pmatrix}}
\newcommand{\epm}{\end{pmatrix}}
\newcommand{\barr}{\begin{displaymath}\begin{array}{cccc}}
\newcommand{\earr}{\end{array}\end{displaymath}}
\newcommand{\barrl}{\begin{displaymath}\begin{array}{lcl}}
\newcommand{\earrl}{\end{array}\end{displaymath}}
\newcommand{\barl}{\begin{displaymath}\begin{array}{l}}
\newcommand{\earl}{\end{array}\end{displaymath}}
\newcommand{\bxym}{ \begin{displaymath}\xymatrix }
\newcommand{\exym}{\end{displaymath}}
\newcommand{\bcd}{\begin{center}\begin{tikzcd}}
\newcommand{\ecd}{\end{tikzcd}\end{center}}

\newcommand{\ev}{\operatorname{ev}}
\newcommand{\fgt}{\operatorname{fgt}}

\newcommand{\Mcal}{\mathcal{M}}
\newcommand{\Dcal}{\mathcal{D}}
\newcommand{\Ecal}{\mathcal{E}}

\newcommand{\Fcal}{\mathcal{F}}
\newcommand{\Lcal}{\mathcal{L}}

\newcommand{\cchern}{\mathrm{c}}
\newcommand{\ol}[1]{\overline{#1}}

\newcommand{\RR}{\mathbb{R}}

\DeclareMathOperator{\spec}{Spec}

\theoremstyle{definition}
\newtheorem{thm}{Theorem}[section]

\newtheorem{lemma}[thm]{Lemma}
\newtheorem{prop}[thm]{Proposition}
\newtheorem{cor}[thm]{Corollary}
\newtheorem*{teo*}{Theorem}

\newtheorem{notation}[thm]{Notation}

\newtheorem{innercustomthm}{Theorem}
\newenvironment{customthm}[1]
  {\renewcommand\theinnercustomthm{#1}\innercustomthm}
  {\endinnercustomthm}

\theoremstyle{definition}
\newtheorem{example}[thm]{Example}

\newtheorem{dfn}[thm]{Definition}
\newtheorem{definition}[thm]{Definition}

\newtheorem{remark}[thm]{Remark}

\newtheorem*{rem}{Remark}
\newtheorem*{aside*}{Aside}

\newlist{steps}{enumerate}{1}
\setlist[steps, 1]{label = Step (\arabic*):}

\newcommand{\thismonth}{\ifcase\month % case 0 --- impossible!
  \or January\or February\or March\or April\or May\or June%
  \or July\or August\or September\or October\or November%
  \or December\fi}

\title[Curve counting in genus one: elliptic singularities {\it \&}  relative geometry]{Curve counting in genus one: \\ elliptic singularities {\it \&} relative geometry}
\author{Luca Battistella, Navid Nabijou and Dhruv Ranganathan}

\begin{document}

\begin{abstract}
We construct and study the reduced, relative, genus one Gromov--Witten theory of very ample pairs. These invariants form the principal component contribution to relative Gromov--Witten theory in genus one and are relative versions of Zinger's reduced Gromov--Witten invariants. We relate the relative and absolute theories by degeneration of the tangency conditions, and the resulting formulas generalise a well-known recursive calculation scheme put forward by Gathmann in genus zero. The geometric input is a desingularisation of the principal component of the moduli space of genus one logarithmic stable maps to a very ample pair, using the geometry of elliptic singularities. Our study passes through general techniques for calculating integrals on logarithmic blowups of moduli spaces of stable maps, which may be of independent interest.
\end{abstract}

\maketitle

\appendixtitletocoff
\tableofcontents

\setcounter{section}{0}

\section*{Introduction}

\footnotetext{\emph{2020 Mathematics Subject Classification.} Primary: 14N35. Secondary: 14H10.\\ \emph{Keywords:} enumerative geometry, logarithmic Gromov--Witten theory, elliptic singularities, quantum Lefschetz.}

\noindent The Kontsevich space of stable maps from genus zero curves to projective space exhibits remarkable geometry -- it is a smooth orbifold with a ``self-similar'' normal crossings boundary. A combination of these facts yields a very satisfactory understanding of Gromov--Witten theory in genus zero. 

The higher genus situation is more delicate, and substantial theory has been developed in recent years to treat the genus one case, both for its intrinsic geometry, and for higher genus insights. This began over a decade ago with pioneering work of Li, Vakil, and Zinger~\cite{redgone,VZ,LZ,lz2,zingerstvsred,zingred}, and was recently reinvigorated using ideas from singularity theory~\cite{BCM18,HL,SMY1,VISC} and logarithmic geometry~\cite{RSPW,RSPW2,BC20}. The result is a \textbf{reduced} Gromov--Witten theory, which removes degenerate contributions from the ordinary theory. Fundamental ingredients in Gromov--Witten theory are torus localisation, relative stable maps, and degeneration formulas. While the Atiyah--Bott localisation for the reduced invariants was used to great effect by Zinger in his proof of the BCOV conjectures, the remaining components -- relative invariants and the degeneration formula -- have not been developed in this setting. Our work provides these as new {structural} tools in reduced genus one Gromov--Witten theory\footnote{Further motivation to consider the genus one case specifically comes simply from the fact that for Calabi--Yau geometries in dimension larger than $3$, Gromov--Witten theory vanishes in genera greater than one.}. 

The essential starting point for our relative theory is the early work on relative Gromov--Witten theory by Gathmann. However, modern tools are required in our setting. The conceptual framework and the degeneration formalism itself are obtained by combining tropical geometry and logarithmic Gromov--Witten theory, as developed by Abramovich--Chen--Gross--Siebert, with the recent understanding of elliptic singularities and their interactions with logarithmic structures. The latter sections of this paper are devoted to demonstrating that the new ingredients can be combined in a concrete fashion, using tautological vector bundles and their Chern classes.

\subsection{Results} Our main construction yields a \textbf{reduced relative} Gromov--Witten theory for elliptic curves in $\mathbb P^m$ with prescribed tangency along a hyperplane, and analogous invariants for smooth hyperplane section pairs $(X,Y)$. The reduced relative theory differs from the relative theory in a parallel fashion to how Li--Zinger's reduced Gromov--Witten theory differs from the standard one. 

The geometric content is the following theorem, which provides a smooth and proper moduli space compactifying the space of elliptic curves in $\mathbb P^m$ with prescribed contact order data. Let $\mathcal M_{1,\alpha}^\circ(\mathbb P^m|H,d)$ be the moduli space of positive degree maps from smooth curves of genus one to $\mathbb P^m$ with dimensionally transverse contact order $\alpha\in\mathbb Z^n_{\geq 0}$ along $H$. A precise version of the following theorem requires tropical and logarithmic language, and is therefore stated precisely in \S \ref{section construction}.

\begin{customthm}{A}\label{thm: desingularization}
There exists a proper Deligne--Mumford stack $\VZ_{1,\alpha}(\mathbb P^m|H,d)$ that is logarithmically smooth of expected dimension, and contains $\mathcal M_{1,\alpha}^\circ(\mathbb P^m|H,d)$  as a dense open subset. A point in this moduli space determines commutative diagrams of maps out of genus one curves:
\[
\begin{tikzcd}
C\arrow{r} & \overline C_1\arrow{d}\arrow{r} & \overline C_2\arrow{d} \\ 
& \mathbb P^m[s]\arrow{r} & \mathbb P^m,
\end{tikzcd}
\]
where $C$ is nodal, $\overline C_1$ and $\overline C_2$ are Gorenstein, $C\to \PP^m[s]$ is a logarithmic map to an $s$-fold expansion of $\PP^m$ along $H$, and the vertical maps do not contract any subcurve of genus one.
\end{customthm}

%Intuitively, non-degeneracy is the condition that the maps in question are non-constant on some branch of the minimal genus one subcurve, see  \S \ref{section construction} for %details.

%The moduli space above is constructed by identifying two closed conditions that must be satisfied by maps to the pair $(\mathbb P^m,H)$. Each condition requires the stable map to factor through an elliptic singularity, and the conditions together are required for the relative and absolute reduced theories to be related appropriately.  {An analysis of the deformation theory of such maps leads to Theorem~\ref{thm: desingularization}}.

The information involving singularities adds substantial complexity to the theory, and this motivates our next result, which shows that calculations are nonetheless possible. 

\begin{customthm}{B}\label{thm: recursion}
Given a smooth pair $(X,Y)$ with $Y$ very ample, there is an explicit recursive algorithm to calculate
\begin{enumerate}
\item the (restricted) reduced genus one Gromov--Witten invariants of $Y$;
\item the reduced genus one relative Gromov--Witten invariants of $(X,Y)$; 
\item the (restricted) reduced genus one rubber invariants of $\mathbb{P}=\PP_Y(\operatorname{N}_{Y|X} \oplus \OO_Y)$
\end{enumerate}
from the absolute genus zero and absolute reduced genus one invariants of $X$.
\end{customthm}

A conceptual consequence is that the absolute and relative reduced genus one theories, together with the genus zero theory, form a \textit{self-contained} system of invariants, i.e. the non-reduced relative theory is not required to perform a degeneration computation in the reduced theory. %The second is that the theorems together give a Gromov--Witten theoretic perspective on the ``main component'' calculations performed in the PhD thesis of Vakil~\cite{Vre}. The first theorem generalises the definition of the invariants in loc. cit. while the second generalises the calculation scheme itself. 

The reconstruction algorithm generalises that of Gathmann~\cite{Ga}. A higher genus reconstruction for the ordinary theory is given by Maulik and Pandharipande~\cite{MaulikPandharipande}, but follows a different strategy; localisation is never used in the analysis here, nor in that of Gathmann. 

We pursue a strategy laid out by Vakil and Caporaso--Harris~\cite{CH98,Vre}. We express the locus of maps with degenerate tangency orders in terms of Chern classes of tautological bundles, and then describe that locus in terms of moduli spaces with smaller invariants. Novelties are introduced in both steps. In the first, we describe the locus of degenerate maps as the zero locus of a logarithmic line bundle, arising from a piecewise linear function with a tropical moduli theoretic description. The use of tropical techniques simplifies Gathmann's calculations.

For the second step, we contend with the interaction of the relative splitting formula with the structure of elliptic singularities, which is new (\S \ref{subsection C0 splitting}). The factorisation condition -- central to the constructions in~\cite{RSPW,RSPW2} -- is expressed as a tautological class, leading to Theorem~\ref{thm: recursion}.

\subsection{Context and techniques} The present work is part of a larger ongoing program relating to the role of curve singularities and alternate compactifications in Gromov--Witten theory that go beyond the well-studied geometry of stable maps from nodal curves~\cite{BC20,BCM18,Boz19,HLN,RSPW,RSPW2}. The motivation for this direction is that such alternative curve counting theories can often be more {enumerative} in nature and satisfy Lefschetz section theorems. Degeneration methods are a crucial tool in Gromov--Witten theory, and our results point to a fruitful use of these methods in this direction. 

The computations performed in this paper may be of use more broadly in logarithmic Gromov--Witten theory. Degenerate moduli spaces are constructed from simpler spaces of maps, but the geometry is substantially more delicate than the standard situation, where the strata are simply fibre products. First, one must describe the moduli space of maps out of elliptic singular curves in terms of maps out of the branches of this curves. We describe this in terms of natural tautological classes. Second, the degenerate strata of moduli spaces are not simply fibre products of spaces with smaller numerical invariants, but compactified torus bundles over these strata. The latter phenomenon is a constant presence in logarithmic Gromov--Witten theory~\cite{AW,R19}. The text provides a model procedure that uses tropical geometry to calculate with these torus bundles. 

A key step in our recursion is a description of the locus of maps with higher than prescribed tangency, which was identified by Gathmann. We realise the locus as the vanishing locus of a section of a line bundle that comes from tropical geometry — from a piecewise linear function on the fan/tropicalization. The systematic understanding of logarithmic line bundles arising in this fashion is expected to play an important role in logarithmic enumerative geometry. 

The geometric ideas in Gathmann's recursion are clear, but are implemented in the space of stable maps, and do not formally fit within the logarithmic context On the way to our main results, we reinterpret Gathmann's idea of increasing tangency in logarithmic Gromov--Witten theory. %This is done by the introduction of \emph{fictitious markings}, and an identification of the space of maps with higher contact order with a boundary stratum inside that of lower contact order and more (fictitious) markings. 

%Our methods require us to work simultaneously with several variants of the space of relative/logarithmic stable maps, with fixed and expanded target. The conceptual features of one space are often computational bugs, and vice versa. Kato's perspective on logarithmic blowups as subfunctors lies at the heart of these ideas. We believe that the techniques developed in the genus one case here will be used repeatedly in logarithmic Gromov--Witten theory calculations. The necessity of working with multiple virtual birational models in logarithmic enumerative geometry is apparent in recent and ongoing work of a number of researchers, and we hope that the methods developed here will be useful more broadly. 

\subsection{Future directions} There are two natural directions moving forward -- generalisations on the curve side and on the target side. The results of~\cite{RSPW} suggest the existence of a reduced higher genus Gromov--Witten theory formed by replacing contracted elliptic components with singularities, and this is reaffirmed by recent work of Bozlee~\cite{Boz19}. More general singularities are introduced and studied in work of Battistella--Carocci~\cite{BC20}. We expect the relative geometry and splitting techniques studied here to be useful in controlling these more exotic invariants.%, where few computations have been done. 

%Our analysis leads naturally to the study of the main component of the double ramification cycle for target manifolds in genus one. The virtual geometry of this was studied in recent work of Janda, Pandharipande, Pixton, and Zvonkine~\cite{DRCBundle}. Even in genus one, an understanding of the main component contributions of the double ramification cycle for target manifolds seems to be lacking. Our main algorithm provides a method for calculating integrals over this class, but it would be more natural to have a formula for it.

On the target side, it would be natural to generalise the method of degenerating tangency conditions to higher rank logarithmic structures. The techniques developed in \cite{MaximalContacts} are expected to play a role. When the target is a toric pair the problem is tied in with the enumerative geometry of well-spaced tropical curves~\cite{LR18,RSPW2,Speyer}. %We will return to these questions in future work. 

\subsection*{Acknowledgements} We have benefited from fruitful conversations with friends and colleagues, including Pierrick Bousseau, Francesca Carocci, Cristina Manolache, Davesh Maulik, Keli Santos-Parker, and Jonathan Wise. The authors worked on this project during visits to MPIM-Bonn, MIT, University of Cambridge, and the University of Glasgow, and thank these institutions for ideal working conditions and support. The text has been improved substantially by the careful reading and comments of a referee, to whom we are very grateful. 

\subsection*{Funding} L.B. and N.N. would in addition like to thank Imperial College London, the LSGNT, the Royal Society, 3CinG, and the MSRI program ``Enumerative Geometry beyond Numbers''. N.N. was partially supported by EPSRC grant EP/R009325/1.  L.B. is supported by the Deutsche Forschungsgemeinschaft (DFG, German Research Foundation) under Germany’s Excellence Strategy EXC-2181/1 - 390900948 (the Heidelberg STRUCTURES Cluster of Excellence).

%\part{ The moduli space and its deformation theory}

\section{Constructions and logarithmic smoothness}\label{section construction}
\noindent We review the basic tropical structures in the theory of logarithmic curves and maps. A detailed treatment is given by~\cite[Section~3]{CavalieriChanUlirschWise}; the following paragraphs are meant as a reminder. 

\subsection{Curves and tropical curves} A nodal curve $C$ over $\spec \mathbb C$ has an associated dual graph $\Gamma(C)$. Given a toric monoid $P$, there is a logarithmic point $\spec (P \to \CC)$ which may be thought of as remembering the information of the embedding of the torus-fixed closed point into the toric variety $\spec \mathbb C[P]$. For a logarithmically smooth curve $C$ over $\spec (P \to \CC)$, the dual graph $\Gamma(C)$ is enhanced: given an edge $e$ of $\Gamma(C)$, the logarithmic structure keeps track of a generalised ``edge length'' $\ell_e$, which is an element of $P$. These data can be repackaged into a cone complex $\plC$ together with a map to the dual cone $\sigma_P$
\[
\pi: \plC\to \sigma_P.
\]
A fibre of $\pi$ is a metric space enhancing the topological space $\Gamma(C)$, with edges given lengths in $\RR_{\geq 0}$. Over the interior of the cone $\sigma_P$ these edge lengths are nonzero, and the fibres are homoemorphic to $\Gamma(C)$. Over the faces of $\sigma_P$, the fibres are viewed as edge contractions of $\Gamma(C)$. 

The discussion is globalized in~\cite{CavalieriChanUlirschWise}, and the construction can be carried out for the universal curve over the moduli space of logarithmic curves. There is a stack $\mathfrak M_{g,n}^{\mathrm{trop}}$ over the category of cone complexes equipped with a universal curve 
\[
\plC\to \mathfrak M_{g,n}^{\mathrm{trop}}.
\]
It is common to view $\mathfrak M_{g,n}^{\mathrm{trop}}$ as playing the same role for $\mathfrak M_{g,n}$ as a fan plays for a toric variety. Two specific such roles are relevant:
\begin{enumerate}
\item A subdivision of $\mathfrak M_{g,n}^{\mathrm{trop}}$ gives rise to a birational modification of $\mathfrak M_{g,n}$.
\item A piecewise linear function on $\mathfrak M_{g,n}^{\mathrm{trop}}$ furnishes a Cartier divisor on $\mathfrak M_{g,n}$ with a section.
\end{enumerate}

\subsection{Maps and tropical maps}The discussion carries over to the setting of logarithmic maps. Consider a logarithmic curve $C$ and a logarithmic map to a smooth pair $(X,Y)$ over a logarithmic scheme $S$, with $Y$ connected. For simplicity, assume that $S$ is a $P$-logarithmic point, as above. There is an associated \textbf{tropicalization}, i.e. a family of tropical curves $\plC \to \sigma_P$, together with a map of cone complexes
\[
\plC\to \RR_{\geq 0}.
\]
Here $\RR_{\geq 0}$ is the tropicalization (or ``fan'') of the smooth pair $(X,Y)$. A fibre over $\sigma_P$ encodes a metric enhancement of the dual graph, together with a piecewise linear map to $\RR_{\geq 0}$. There is again a cone stack representing the moduli problem of tropical maps; the construction differs cosmetically from the one presented in~\cite{CavalieriChanUlirschWise}, and we leave the details to the reader; see also~\cite{GrossSiebertLog}.

\subsection{Stable maps and elliptic singularities} We provide a summary of the approach of~\cite{RSPW}. If $C$ is a Gorenstein curve of genus one and $f:C\to \mathbb P^m$ is a morphism, the ampleness of the tangent bundle of $\mathbb P^m$ ensures that if $f$ is non-constant on the minimal subcurve of $C$ of arithmetic genus one, then the deformations of $f$ are unobstructed; in fact, the deformations \textit{over the stack of curves} are also unobstructed. 

It is natural to attempt to resolve the moduli space of stable maps by contracting the genus one subcurve on which the map is constant. Whilst the moduli space of Gorenstein genus one curves is not smooth, a more sophisticated moduli problem, consisting of nodal curves $C$ together with a \textit{chosen contraction} to a Gorenstein curve $\bar{C}$, might be expected to resolve these singularities. The formation of this moduli problem requires logarithmic geometry, and furnishes a moduli space with a universal \textit{nodal curve}, universal \textit{elliptically singular} curves, and contractions from the former to the latter~\cite[Section~3]{RSPW}. Once constructed, the corresponding moduli of stable maps out of a Gorenstein curve is also resolved. Maps from nodal genus one curves may have stable limits that contract genus one subcurves, but the introduction of Gorenstein singularities -- with an appropriate stability condition -- ensures properness without allowing contracted genus one components. 

The main target space in the present work is $(\mathbb P^m,H)$, and ideas parallel to those above can be expected to apply. However, the deformation theory of logarithmic maps is more delicate: the Euler sequence for this pair shows that the logarithmic tangent bundle is a quotient of $\mathcal O\oplus\mathcal O(1)^{\oplus m}$, and additional obstructions appear due to the trivial factor. %In simple terms: there exist logarithmic stable maps to the pair $(\mathbb P^m,H)$ with obstructed logarithmic deformations, whose underlying stable map to $\PP^m$ is unobstructed. 

\subsection{Singularities and moduli of attachments}\label{S:ellsing} Isolated Gorenstein singularities of genus one were classified by Smyth~\cite{SMY1}: for every fixed number of branches $m$, there is a unique isomorphism type of curve germ. For $m=1$ and $2$ these are the cusp ($A_2$) and tacnode ($A_3$) respectively; for $m\geq 3$ they are the union of $m$ general lines through the origin of $\Aaff^{m-1}$ -- the \emph{elliptic $m$-fold points}.

The problem of constructing an elliptic singularity from its pointed normalisation is delicate; it was studied in van der Wyck's thesis \cite{vdW}. Recall that the seminormalisation $C^\text{sn}$ of a curve $C$ has the same underlying topological space, but its singularities consist of rational $m$-fold points only, i.e. they are locally isomorphic to the union of the axes in $\Aaff^m$.

\begin{lemma}\label{lem:pushout}
 The seminormalisation $\nu\colon C^\text{sn}\to C$ of an elliptic $m$-fold point is the collapse of a general tangent vector $v$ at the rational $m$-fold point.
\end{lemma}

\noindent This is equivalent to \cite[Lemma 2.2]{SMY1}, where he describes the seminormalisation of an elliptic singularity as the data of a hyperplane in the Zariski cotangent space not containing the coordinate lines. We use the term \emph{moduli of attaching data} \cite[\S 2.2]{SMY2} for the extra data needed to produce a singularity from its pointed normalisation.

\begin{cor}
 The moduli of attaching data of an elliptic $m$-fold point is isomorphic to $\mathbb G_{\text{m}}^m$.
\end{cor}
\noindent This torus may be compactified by allowing \emph{sprouting} on the boundary: a branch $R_i$ of the singularity is replaced with a strictly semistable $\PP^1$ with the singularity at $0$ and $R_i$ glued nodally at $\infty$.

There are other ways to look at these data; we recall a construction from linear algebra. Let $V=\sum_{i=1}^mL_i$ be a presentation of a vector space as the direct sum of $m$ lines. The datum of a general line $\ell$ in $V$ (or a general hyperplane in $V^\vee$) is equivalent to the data of compatible isomorphisms $\theta_{ij} \colon L_i \cong L_j$. Given a general line $\ell$, a vector $v\in\ell$ can be written as:
\[v=a_1e_1+\ldots+a_me_m,\qquad \text{with $a_i\neq 0$ for all $i$}\]
for a fixed basis $\{e_i\in L_i\}$ of $V$. We can then define $\theta_{ij}$ by
\[\theta_{ij}(e_i)=(a_j/a_i)\cdot e_j\]
and it is easy to see that it does not depend on any of the choices. On the other hand, given the $\theta_{ij}$, the line $\ell$ can be taken to be the graph of the morphism:
\[(\theta_{12},\ldots,\theta_{1m})\colon L_1\to L_2\oplus\ldots\oplus L_m.\]
Thus the attaching data can be viewed as specifying an identification of tangent spaces of the pointed normalisation at the singularity.

Lemma \ref{lem:pushout} yields a characterisation of maps from an elliptic singularity $C$, which is used in \S \ref{section reduced splitting}:
\begin{cor}\label{cor:maps_from_elliptic_sing}
 A map $C\to X$ is the same as a map $C^\text{sn}\to X$ such that its differential sends $v \mapsto 0$.
\end{cor}
\noindent Line bundles on an elliptic $m$-fold point admit a similar description, see e.g. \cite[Lemma 7.5.12]{LiuQ}.
\begin{lemma}\label{lem:PicE}
 If $C$ is a curve of arithmetic genus one with an elliptic $m$-fold point then $\operatorname{Pic}^0(C)\cong\mathbb G_{\text{a}}$.
\end{lemma}

\subsection{Absolute geometry: background} Let $\mathfrak M_{1,n}$ be the logarithmic algebraic stack of $n$-pointed prestable curves of genus one. We recall the logarithmic moduli spaces of genus one curves constructed in~\cite[\S 2 \& \S 4]{RSPW}. 

Consider a logarithmically smooth curve $C$ over a logarithmic geometric point $S$ with monoid $P$ and let $\pi\colon \plC\to \sigma_P$ denote its tropicalization. Given a point $t$ in $\sigma_P$ we obtain a metric graph as explained above. Since the curve in question has genus one, there is a minimal genus one subgraph -- either a vertex of genus one or a cycle in the graph. We refer to this as the \textbf{core}, and denote it $\plE$.

Given a point $s$ with $\pi(s)$ equal to $t$, there is a well-defined distance in the metric graph $\pi^{-1}(t)$ from $s$ to the core. 
For each vertex $v$ of the dual graph $\Gamma(C)$, there is a well-defined function $\lambda(v): \sigma_P\to \RR_{\geq 0}$, defined on the base of the tropical family. We view $\boldsymbol{\lambda}$ as a piecewise linear function $\lambda\colon\plC\to\RR_{\geq 0}\times\sigma_P$, see Figure \ref{fig:lambdaexample}. Given any linear function $\delta$ on $\sigma_P$, say that $\delta$ is \textbf{comparable} to $\lambda(v)$ if a global order relation, either $\lambda(v)\leq \delta$ or $\lambda(v)\geq \delta$, that holds at all points of $\sigma_P$. 

Piecewise linear functions on $\plC$ are naturally identified with global sections of the characteristic sheaf $\overline M_C$; linear functions on $\sigma_P$ are naturally identified with global sections of the characteristic sheaf $\overline M_S$, and can be pulled back to $C$. Given a logarithmically smooth curve $C/S$, conditions involving the comparability of piecewise linear functions can be imposed at the level of geometric points; they determine subdivisions of $\sigma_P$, and, in turn, logarithmic modifications of $S$. %This leads to the following definition. 

\begin{definition}\label{def: alignment}
A \textbf{centrally aligned curve} is a pair $(C/S,\delta)$, where $\delta\in\overline{M}_S$, such that:
\begin{enumerate}
    \item the section $\delta$ is comparable to $\lambda(v)$ for all vertices $v$ of $\plC$;
    \item for any pair of vertices $v$ and $w$ of distance less than $\delta$ from the core, the sections $\lambda(v)$ and $\lambda(w)$ are comparable.
\end{enumerate}
The order relations among the vertex position functions is referred to as an \textbf{alignment}. 
\end{definition}
For a tropical curve  family $\sqC$, each vertex is either inside, outside, or on the circle of radius $\delta$ regardless of the point on the base; the vertices inside the circle are ordered by distance to the core.

\begin{figure}
 \begin{tikzpicture} 
  \tikzset{cross/.style={cross out, draw=black, fill=none, minimum size=2*(#1-\pgflinewidth), inner sep=0pt, outer sep=0pt}, cross/.default={3pt}}
  
  \foreach \x in {(.5,2.5),(3.3,2.5),(3.65,2.5),(3.7,2.5),(4.4,2.5),(4.65,2.5),(-1.7,2.5)}
  \draw \x node[cross,rotate=45] {};
  
  \foreach \x in {(-1,1),(3.5,1.5),(3.2,2),(-2.5,2.5),(-.5,1.5),(1.5,3.5),(4.8,3),(5.8,3.5),(7,0),(-2,2.8),(4,3.2)}
  \draw[fill=black] \x circle (2pt);
  
  \draw[gray] (0,0) node[left]{$\plE$} -- (1,.2) -- (2,0) -- (1,-.2) -- (0,0); %core
  
  \draw (0,0) -- (-2.5,2.5) -- (-2.7,4) (-2.5,2.5) -- (-2.3,4) (-1,1)--(1.5,3.5)--(1.3,4) (1.5,3.5)--(1.7,4) (-.5,1.5) -- (-1.7,2.5) -- (-2,2.8) --(-2.2,4) (-2,2.8) --(-1.8,4);
  \draw (1,.2) -- (4,3.2) -- (3.9,4) (4,3.2) -- (4.1,4);
  \draw (1,-.2) -- (3.2,2) -- (5,4) (3.2,2) -- (5.2,4);
  \draw (2,0) -- (3.5,1.5) -- (4.8,3) -- (5.4,4) (4.8,3) -- (5.6,4) (3.5,1.5) -- (5.8,3.5) -- (5.8,4) (5.8,3.5) -- (6,4);
  
  \draw[->] (7,0) -- (7,4); %real axis
  \draw[dashed] (-2.7,1) -- (7.2,1) (-2.7,1.5) -- (7.2,1.5) (-2.7,2) -- (7.2,2); %heights
  \draw[blue,dashed]  (-2.7,2.5) -- (7.2,2.5) node[right]{$\delta$};
  \draw [->] (6,1.55) --node[above]{$\lambda$} (6.8,1.55); %arrow-lambda
   
  \foreach \x in {(0,0),(1,.2),(1,-.2),(2,0)}
  \draw[fill=gray] \x circle (2pt);
 \end{tikzpicture}
\caption{The function $\lambda$ (distance from the core) and its cutoff value $\delta$.}
\label{fig:lambdaexample}
\end{figure}
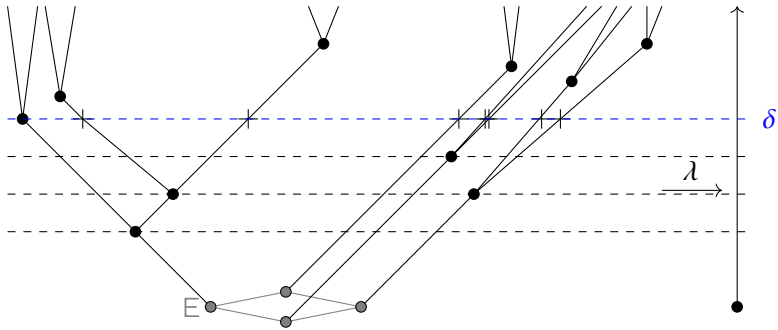

The moduli stack $\mathfrak M_{1,n}^{\mathrm{cen}}$  of centrally aligned curves is a logarithmic algebraic stack in the smooth topology, and $\mathfrak M_{1,n}^{\mathrm{cen}}\to\mathfrak M_{1,n}$ is a logarithmic modification. The main construction of~\cite{RSPW} canonically associates to any centrally aligned family of curves $\mathcal C_S$ a partial destabilisation $\widetilde{\mathcal C}_S$, introducing $2$-pointed components, and a contraction
\[
\widetilde{\mathcal C}_S\to \overline{\mathcal C}_S,
\]
where $\overline{\mathcal C}_S$ may contain a Gorenstein elliptic singularity \cite{SMY1}. Intuitively, this is achieved by interpreting the exceptional directions of $\mathfrak M_{1,n}^{\mathrm{cen}}\to\mathfrak M_{1,n}$ as moduli of attaching data.

Given a family of tropical curves, the function $\delta$ provides a non-negative real number for each point on the base. We consider the \textbf{circle} of points on each fibre whose distance from the core is equal to $\delta$. A partial destabilization, or ``bubbling'', introduces $2$-pointed rational curves into the curve by adding vertices to the dual graph at the points that lie on this circle (the crosses at height $\delta$ in Figure \ref{fig:lambdaexample}); we denote this new tropical curve by $\widetilde{\plC}$. The open disc of radius $\delta$, i.e. the locus where $\lambda<\delta$, is then contracted to a Gorenstein singularity in $\overline{\mathcal C}_S$. The number of branches is equal to the number of vertices of $\widetilde{\sqC}$ which lie on the circle. %This data uniquely determines the singularity.

The space of stable maps $\widetilde\VZ_{1,n}(\PP^m,d)$ from the universal nodal curve $\mathcal C$ over the stack $\mathfrak M_{1,n}^{\mathrm{cen}}$ is Deligne-Mumford and proper, with projective coarse moduli space. A compatibility condition between the alignment and the map is required, namely that the circle of radius $\delta$ passes through \emph{at least one component of positive degree}. A stable map in this space is said to satisfy the \textbf{factorisation condition} if the map $\widetilde{\mathcal C}\to \mathcal C\to \mathbb P^m$ factors through the contraction $\widetilde{\mathcal C}\to \overline{\mathcal C}$. Maps that are not constant on any genus one subcurve of $\mathcal C$ automatically satisfy this condition; indeed, the contraction is an isomorphism in this case.
 
\begin{thm}[{\cite[Theorem B]{RSPW}}]
The substack $\VZ_{1,n}(\mathbb P^m,d)$ of $\widetilde\VZ_{1,n}(\PP^m,d)$ of maps to $\mathbb P^m$ that satisfy the factorisation condition is smooth, proper, and has the expected dimension.
\end{thm}

\begin{remark}[Constant maps]\label{rem: constant-maps}
If the degree of the map is at most $1$, maps from smooth genus one curves to $\PP^m$ are constant and the interior of the space of stable maps is empty; the compactification by stable maps is non-empty. The space of maps satisfying the factorisation condition is also empty. %We do not indulge on this any further. 
\end{remark}

The remainder of this section lifts this construction to the category of relative maps to $(\mathbb P^m,H)$. The core result shows that two factorisation conditions remove the obstructions. 

\subsection{Relative geometry: compactification} Fix a hyperplane $H\subseteq \mathbb P^m$. Let $\alpha$ be a partition of the degree $d>0$. Consider the moduli space $\mathcal M_{1,\alpha}^\circ(\mathbb P^m|H,d)$ of maps from smooth genus one curves $C\to \mathbb P^m$ that meet $H$ at finitely many marked points with vanishing orders given by the partition $\alpha$. This is a smooth, but non-proper Deligne--Mumford stack. 

We compactify the space described and then desingularise it. For compactification, we begin with Abramovich--Chen--Gross--Siebert's logarithmic stable maps, and then work with expanded variants~\cite{AbramovichChenLog,ChenLog,GrossSiebertLog,KimLog}.

Let $\PP^m$ be endowed with the divisorial logarithmic structure induced by $H$. The moduli space $\overline{\mathcal{M}}^{\operatorname{log}}_{1,\alpha}(\mathbb P^m|H,d)$ is a fibred category over logarithmic schemes, whose fibre over $(S,M_S)$ is the groupoid of logarithmically smooth curves of genus one over $(S,M_S)$ equipped with a logarithmic map to $\mathbb P^m$ of degree $d$ and contact order $\alpha$. This category is representable by a proper algebraic stack with logarithmic structure, parameterising minimal objects~\cite{ChenLog}. 

%%%%%
\begin{comment}
There is a representable finite logarithmic morphism to the Kontsevich space, forgetting the logarithmic structure on the target:
$$
\overline{\mathcal M}^{\mathrm{log}}_{1,\alpha}(\mathbb P^m|H,d) \to \overline{\mathcal M}_{1,n}(\mathbb P^m,d).
$$
\end{comment}
%%%%%%%%%%%%%%

\begin{comment}
As in the absolute setting, we consider the space of logarithmic morphisms from the universal curve $\mathcal{C} \to \mathfrak{M}_{1,n}^{\mathrm{cen}}$ to the logarithmic scheme $(\PP^m,H)$. We impose a compatibility condition, which says that the radius must pass through a vertex of positive degree in $\PP^m$ and contain only vertices of degree zero in its strict interior. In this way we obtain a logarithmic modification
\begin{equation*} \widetilde\VZ_{1,\alpha}(\PP^m|H,d) \to \ol\Mcal^{\mathrm{log}}_{1,\alpha}(\PP^m|H,d)\end{equation*}
which we refer to as the \textbf{moduli space of centrally aligned maps to $(\PP^m,H)$}. It is a logarithmic modification of $\overline{\mathcal M}^{\mathrm{log}}_{1,\alpha}(\mathbb P^m|H,d)$. We note that by definition this space maps down to $\widetilde\VZ_{1,n}(\PP^m,d)$, by forgetting the logarithmic structure of the target.
\end{comment}

\subsection{Relative geometry: expansions} 

\begin{comment}Our next task is to pick out a non-singular principal component in $\widetilde{\VZ}_{1,\alpha}(\mathbb P^m|H,d)$. The principal component of this space, consisting of the closure of the space of maps from nonsingular curves, maps into the principal component of $\widetilde\VZ_{1,n}(\PP^m,d)$. Indeed, smoothable logarithmic maps are, in particular, smoothable as ordinary maps. An additional condition is required to isolate the principal component of the space of logarithmic maps. \end{comment}

In order to make the obstruction theory more geometric, we expand the target. To elucidate the connection with the static target, consider a logarithmic stable map $[C\to (\mathbb P^m,H)]$ over $\mathrm{Spec} \ (\mathbb N\to \mathbb C)$. There is a map of tropicalizations:
\[
\plC\to \mathbb R_{\geq 0}.
\]
Choose a subdivision of $\mathbb R_{\geq 0}$ whose vertices are the images of vertices of $\plC$. Pull this subdivision back to $\plC$ by subdividing it along the preimage of the new vertices of $\mathbb R_{\geq 0}$. Denote the resulting map $\widetilde \plC \to \widetilde{\mathbb R}_{\geq 0}$. These subdivisions induce logarithmic modifications:
\[
\widetilde C\to \mathbb P^m[s],
\]
see~\cite{AW}. The latter is the $s$-times iterated deformation to the normal cone of $H$ in $\mathbb P^m$. Components of the target are in bijection with the vertices in $\widetilde{\mathbb R}_{\geq 0}$. The curve is modified by adding rational components for the newly introduced vertices.

The result is a logarithmic stable map to an expansion $\mathbb P^m[s]$, with a collapse to the main component $\mathbb P^m[s]\to\mathbb P^m$. Maps are considered equivalent if they differ by $\Gm$-scaling the target at higher level. This is a logarithmic enhancement of the theory of relative maps, due to Li \cite{Li1}.

Kim constructs a moduli space of logarithmic maps to expanded degenerations, which on logarithmic points, gives the above construction~\cite{KimLog}. Kim's space is identified with a subcategory of the Abramovich--Chen--Gross--Siebert space (as stacks over logarithmic schemes), and its dual minimal monoids are cones in a subdivision of the dual minimal cones of the unexpanded space, see~\cite[\S~2]{R19}. The map from this space to the logarithmic space is akin to a blowup: it is a logarithmic modification induced by a subdivision of tropical moduli spaces. We typically work with expansions, and use the notation
$\ol\Mcal_{1,\alpha}(\PP^m|H,d)$
for Kim's space of logarithmic stable maps.

\subsection{Relative geometry: alignment} Consider the space of logarithmic morphisms from the universal curve $\mathcal{C} \to \mathfrak{M}_{1,n}^{\mathrm{cen}}$ to the logarithmic scheme $(\PP^m,H)$. As in the absolute case, an alignment can be introduced on the source curve as in Definition~\ref{def: alignment}. The circle associated to the alignment must pass through a vertex of positive degree in $\PP^m$ and contain only vertices of degree zero in its strict interior. We obtain a logarithmic modification of the Abramovich--Chen--Gross--Siebert space. By means of additional subdivisions of the tropical target and source, as outlined in the previous section, we obtain a logarithmic modification of Kim's space with expansions:
\begin{equation*} \widetilde\VZ_{1,\alpha}(\PP^m|H,d) \to \ol\Mcal_{1,\alpha}(\PP^m|H,d)\end{equation*}
which we refer to as the \textbf{moduli space of centrally aligned maps} to $(\PP^m,H)$. This space parametrises logarithmic maps to the expanded target with an alignment which is compatible with the map. There is a morphism
\begin{equation*} \widetilde\VZ_{1,\alpha}(\PP^m|H,d) \to \widetilde\VZ_{1,n}(\PP^m,d)\end{equation*}
obtained by forgetting the logarithmic structure on the target and collapsing the map.

\subsection{Relative geometry: factorisation}\label{subsection factorisation} Let $[C\to \mathbb P^m[s]\to \mathbb P^m]$ be a centrally aligned logarithmic map. Recall that $\mathbb P^m[s]$ consists of a union of $\mathbb P^m$ with $s$ copies of the projective bundle $\mathbb{P} = \mathbb P(\mathcal O_H\oplus \mathcal O_H(1))$. The latter components are \textbf{the higher levels}. We will say that a subcurve $D\subseteq C$ \textbf{maps to higher level} if the composite $C\to \mathbb P^m$ maps $D$ into $H\subset \mathbb P^m$. 

Let $D_F\subseteq C$ be the maximal genus one subcurve which is mapped to a higher level and contracted by the map $C\to \mathbb P^m[s]$. Let $\delta_F$ denote the distance from the core to the nearest component of $C \setminus D_F$. Similarly let $D_B\subseteq C$ be the maximal genus one subcurve that is contracted by the collapsed map to $\mathbb P^m$, and let $\delta_B$ be the associated radius. This coincides with the radius of the underlying map to $\mathbb P^m$. Of course, $\delta_F\leq \delta_B$, so $(C,\delta_F)$ is centrally aligned as well.

The datum $(\delta_F,\delta_B)$ determines a destabilisation $\widetilde C$ of $C$ together with successive contractions $\widetilde C\to \overline C_F\to\overline C_B$. The contractions are constructed by~\cite[Section~3]{RSPW}.

%A Gorenstein elliptic singularity has a number of \textbf{branches}. Given a curve $\overline C$ of genus one with such a singularity, define a branch of $\overline C$ to be an irreducible component containing a branch of the singularity. If the curve $C$ is nodal, then its branches will be understood to be the components of its minimal arithmetic genus one subcurve, i.e. its core.
The \emph{core} of a Gorenstein curve of genus one is the minimal genus one subcurve. Any such curve decomposes as the nodal union of the core and a (possibly empty) rational forest. If the curve contains an elliptic $m$-fold point, we say a map is non-constant on the core if at least one branch of the singularity is not contracted.
The following condition identifies smoothable maps. 

\begin{definition}
The map $f\colon C\to \mathbb P^m[s]$ \textbf{factors completely} if 
\begin{enumerate}
\item $f$ factors through $\overline C_F$ and it is non-constant on its core,
\item letting $f_B\colon C\to\PP^m$ denote the composite of $f$ with the collapsing map $\mathbb P^m[s]\to\mathbb P^m$, the map $f_B$ factors through $\overline C_B$ and it is non-constant on its core.
\end{enumerate}
\end{definition}

\begin{rem}
 The factorisation conditions are understood to be satisfied if $C\to \mathbb P^m[s]$ (resp. $C\to \mathbb P^m$) does not contract a genus one subcurve. 
\end{rem}

Let
\begin{equation*} \VZ_{1,\alpha}(\mathbb P^m|H,d) \subseteq \widetilde\VZ_{1,\alpha}(\PP^m|H,d) \end{equation*}
be the stack of maps from centrally aligned curves to expansions that factor completely. If $[C\to \mathbb P^m[s]\to \mathbb P^m]$ is a family of centrally aligned maps over $S$ that factors completely, there is a moduli map $S\to \VZ_{1,n}(\mathbb P^m,d)$, to the principal component of the space of absolute maps. 

\begin{thm}\label{thm: log-smoothness}
The stack $\VZ_{1,\alpha}(\mathbb P^m|H,d)$ is logarithmically smooth, proper, and has the expected dimension.
\end{thm}

Maps out of elliptic curves will be assumed to have degree at least $2$, ensuring that the main component is non-empty, c.f. Remark~\ref{rem: constant-maps}. 

\begin{proof}

There is a forgetful morphism
\[
\nu\colon \VZ_{1,\alpha}(\mathbb P^m|H,d)\to \VZ_{1,n}(\mathbb P^m,d),
\]
remembering the stabilisation of the composite $f_B\colon C\to\PP^m[s]\to\PP^m$. A priori, this is in $\overline{\pazocal M}_1(\PP^m,d)$, but by assumption $(C,\delta_B)$ is centrally aligned, and $f_B$ factors through $\overline C_B$. 

Before embarking, we record that the proof will utilise $\nu$ in the following way: the deformation spaces of objects in the domain of $\nu$ will be cut out from the deformation spaces of their images in the target of $\nu$, by the vanishing of an obstruction class. It will then be shown that this subspace of deformations has everywhere the expected dimension. 

The expanded target $\PP^m[s]$ is logarithmically \'etale over $\PP^m$ with the divisorial structure associated to $H$, so it is equivalent to study the deformation theory of the unexpanded map to $(\PP^m,H)$. The logarithmic morphism $(\PP^m,H)\to \PP^m$ induces an exact sequence:
\[0\to T_{\PP^m}(-\operatorname{log}H)\to T_{\PP^m}\to\OO_H(H)\to 0,\]
relating the absolute logarithmic deformation/obstruction theory of the map with the those relative to the target of $\nu$. By examining the associated long exact sequence, we find that a stable map can be obstructed only when a genus one subcurve is mapped to higher level in such a way that it is contracted by the collapsed map $f_B$. In this case, let $C_0$ denote the core, and $E$ the connected component of $f_B^{-1}(f_B(C_0))$ of arithmetic genus one.

Let $f\colon C\to\PP^m[s]$ over the base $S$ be a possibly obstructed map to the expansion that factors completely, and let $S^\prime$ be a strict square-zero extension of the base.  The deformation problem is local, so we may as well suppose that $\overline M_S^{\text{gp}}$ is a constant sheaf of monoids with stalk $\overline M_{S,s}^{\text{gp}}=:Q$, for a geometric point $s\in S$. Note that sections of $\overline M_S$ such as $\delta_B$ and $\delta_F$ extend automatically due to strictness. Moreover, it is equivalent to deform $f$ or $f_F\colon\overline C_F\to\PP^m_S$, where $\overline C_F$ is given the trivial logarithmic structure around the elliptic singularity.

The map $f_F\colon \overline C_F\to\PP^m[s]$ induces a tropical map $\alpha\colon\plC_F\to\RR_{\geq 0}\times\sigma_Q$; we view this piecewise linear function as a section $\alpha\in H^0(\overline C_F,\overline M_C^{\text{gp}})$. With notation as above, we obtain a trivialisation of the line bundle $\OO_{\overline C_F}(-\alpha)$ associated to $\alpha$, after restriction to $E$. By \cite[Proposition 2.4.1]{RSPW}, the latter is isomorphic to $\OO_E(\sum\mu_i x_i)$, where $\mu_i$ is the slope of $\alpha$ along the edge $e_i$ of $\plC_F$, and $x_i$ the node corresponding to $e_i$ thought of as a smooth marking on $E$. 

For a \textit{strict} deformation of the curve, the piecewise linear function $\alpha$ extends canonically. The deformation problem for $f_F$ relative to $(C^\prime,\delta_B,f_B^\prime)$ reduces to the problem of extending $\alpha$ as a trivialisation of $\OO_E(\sum\mu_i x_i)$. This identifies $\operatorname{Obs}(f/(C,\delta_B,f_B))$ as $H^1(\overline C_F,\OO)$. The latter is a vector space of dimension $1$. Associated to the forgetful morphism $\nu$ at the level of perfect obstruction theories, we find a map:
\begin{equation}\label{eqn:def-obs}\operatorname{Def}(C,\delta_B,f_B)\to \operatorname{Obs}(f/(C,\delta_B,f_B)).\end{equation}
The left-hand side contains  $\operatorname{Def}(E,\mathbf{x})=\operatorname{Ext}^1(\Omega_E(\mathbf x),\OO_E)$, since $E$ is contracted by $f_B$. Since at least one of the $\mu_i$ is non-zero, $\operatorname{Ext}^1(\Omega_E(\mathbf x),\OO_E)$ surjects onto $H^1(\overline C_F,\OO)$ c.f. \cite[Corollary 3.5.3]{RSPW2}. Looking back at \eqref{eqn:def-obs}, given an infinitesimal deformation of $(C^\prime,\delta_B,f_B^\prime)$, it is possible to adjust the curve by a deformation of the points $\mathbf x$ to ensure that the obstruction vanishes.

We put the pieces together. The space $ \operatorname{Obs}(f/(C,\delta_B,f_B))$ always has rank $1$. Next, note that once the collapsed stable map is fixed, there are no \textit{logarithmic} deformations of the map to $\mathbb P^m[s]$: the morphism from logarithmic maps to ordinary maps is finite~\cite[\S~3.7]{ChenLog}. View the logarithmic deformation spaces of $f$ as subspace of the logarithmic deformation spaces of $(C^\prime,\delta_B,f_B^\prime)$, consisting of those deformations with vanishing class in $\operatorname{Obs}(f/(C,\delta_B,f_B))$. The map to this obstruction group is everywhere surjective, and since the data $(C^\prime,\delta_B,f_B^\prime)$ have unobstructed deformations, this subspace is everywhere of the expected dimension. Obstructions to logarithmic deformations of $f$ vanish, and we conclude.
\end{proof}

The theorem guarantees that the space of completely factoring maps to expansions has toric singularities, but one can say more.

\begin{cor}
The logarithmically smooth stack $\VZ_{1,\alpha}(\mathbb P^m|H,d)$ has at worst orbifold singularities, i.e. admits a non-representable cover by a smooth Deligne--Mumford stack.
\end{cor}

%\textcolor{red}{Luca's rewrite of the Corollary: Up to a generalised root construction, $\VZ_{1,\alpha}(\mathbb P^m|H,d)$ is a smooth Deligne-Mumford stack.}

\begin{proof}
Since $\VZ_{1,\alpha}(\mathbb P^m|H,d)$ is logarithmically smooth, we must show that the cones of its tropicalization are simplicial. Consider a logarithmic stable map to an expansion $C\to \mathbb P^m[s]$ without a central alignment. The tropical moduli cone obtained as the dual of the minimal base monoid can be identified with $\mathbb R_{\geq 0}^{s}$, see for instance~\cite[Section~2.2]{ChenDegeneration}. The alignment is an iterated stellar subdivision on tropical moduli spaces, see~\cite[Section 4.6]{RSPW}, and such subdivisions preserve the property of being simplicial. We conclude that the blowup $\widetilde{\VZ}^{\mathrm{exp}}_{1,\alpha}(\mathbb P^m|H,d)$ obtained by centrally aligning Kim's spaces has simplicial cones. The morphism $\VZ_{1,\alpha}(\mathbb P^m|H,d)\to \widetilde{\VZ}^{\mathrm{exp}}_{1,\alpha}(\mathbb P^m|H,d)$ is strict, so the cones of $\VZ_{1,\alpha}(\mathbb P^m|H,d)$ are simplicial.
\end{proof}

\subsection{Rubber variants}\label{subsection rubber} There exists a rubber variant of $\VZ_{1,\alpha}(\mathbb P^m|H,d)$, where the curve is contained in the higher levels, whose construction is analogous. Let $\mathbb P$ denote the projective bundle $\mathbb P(\mathcal O(1)\oplus\mathcal O)$ on $H\cong\mathbb P^{N-1}$. Equip this space with the logarithmic structure coming from the $0$ and $\infty$ sections of the bundle. Consider the moduli space $\VZ^{\leftrightarrow}_{1,\alpha}(\mathbb P|H_0+H_\infty,d)$ of logarithmic maps
\[
C\to \mathbb P[s]\to \mathbb P^{N-1},
\]
which factor completely, and where automorphisms are taken to cover the identity on the map to $\mathbb P^{N-1}$; two maps that differ by a $\mathbb G_{\operatorname{m}}$ translate in the fibre direction are considered equivalent. As previously, the map to the bundle and the map to the base $\mathbb P^{N-1}$ both factor through possibly different singularities. 

\begin{thm}\label{thm rubber log smooth}
The stack $\VZ^{\leftrightarrow}_{1,\alpha}(\mathbb P|H_0+H_\infty,d)$ is logarithmically smooth, proper, and has the expected dimension.
\end{thm}

On the locus of maps from smooth domains, the map from $\VZ^\circ_{1,\alpha}(\mathbb P|H_0+H_\infty,d)\to \VZ^{\leftrightarrow,\circ}_{1,\alpha}(\mathbb P|H_0+H_\infty,d)$ is a $\mathbb G_{\operatorname{m}}$-torsor. One expects that the map on compact moduli spaces is a nodal curve fibration. To prove this, it is convenient to work with the logarithmic multiplicative group and its torsors. The \textbf{logarithmic multiplicative group} $\mathbb G_{\mathrm{log}}$ is the functor on logarithmic schemes whose value on a logarithmic scheme $S$ is the group of global sections $H^0(S,M_S^{\mathrm{gp}})$; it is a proper group functor, and contains $\Gm$ as a subfunctor (from $\OO^\times_S\subseteq M_S^{\mathrm{gp}}$). This functor is representable after a logarithmically \'etale~\cite[Proposition~1]{RW19}.

The morphism $\mathbb P \to \mathbb P^{N-1}$ determines a $\mathbb G_m$-torsor; the inclusion $\mathbb G_m\hookrightarrow \mathbb G_{\mathrm{log}}$, determines a $\mathbb G_{\mathrm{log}}$-torsor denoted $\mathbb P_{\mathrm{log}}$ which is a functor on logarithmic schemes with a logarithmically \'etale cover by a scheme equipped with a logarithmic structure.

%%%%%%%%%%%%%%%
\begin{comment}
%Bundle can have higher levels, so should rephrase this
The bundle $\mathbb P\to \mathbb P^{N-1}$ gives rise to a $\mathbb G_{\operatorname{m}}$-torsor by deleting the zero and infinity sections. Replacing these fibres by their $\mathbb G_{\mathrm{log}}$ compactifications, we obtain a non-representable functor on logarithmic schemes $\mathbb P_{\mathrm{log}}$ and a logarithmically \'etale modification
\[
\mathbb P\to \mathbb P_{\mathrm{log}}.
\]
\end{comment}
%%%%%%%%%%%%%%

\begin{proof}[Proof of Theorem \ref{thm rubber log smooth}]
Consider the stack over logarithmic schemes $\VZ_{1,\alpha}(\mathbb P_{\mathrm{log}},d)$ of stable logarithmic maps that factor completely, noting the factorisation for the map to $\mathbb P_{\mathrm{log}}$ may be imposed after passing to an expansion c.f.~\cite[Section~3.3]{RSPW2}. Stability here coincides with stability for the projection to $\mathbb P^{N-1}$. The proof of Theorem~\ref{thm: log-smoothness} applies here to show that maps to $\mathbb P$ factorising completely are unobstructed. This space is a logarithmically \'etale cover of $\VZ_{1,\alpha}(\mathbb P_{\mathrm{log}},d)$, and logarithmic smoothness of the latter also follows.

The logarithmic multiplicative group $\mathbb G_{\mathrm{log}}$ acts on $\VZ_{1,\alpha}(\mathbb P_{\mathrm{log}},d)$ by translation without fixed points. Tautologically, $\VZ_{1,\alpha}(\mathbb P_{\mathrm{log}},d)$ is a $\mathbb G_{\mathrm{log}}$-torsor over the moduli problem of maps up to this $\mathbb G_{\mathrm{log}}$-translation. As $\mathbb G_{\mathrm{log}}$ is logarithmically smooth, the space $\VZ^{\leftrightarrow}_{1,\alpha}(\mathbb P_{\mathrm{log}},d)$ of maps up to $\Glog$-translation that factor completely is also logarithmically smooth. 

We compare $\VZ^{\leftrightarrow}_{1,\alpha}(\mathbb P_{\mathrm{log}},d)$ and  $\VZ^{\leftrightarrow}_{1,\alpha}(\mathbb P|H_0+H_\infty,d)$. A map to $(\mathbb P,H_0+H_\infty)$ gives rise to a map to $\mathbb P_{\mathrm{log}}$. The morphism
\[
\VZ^{\leftrightarrow}_{1,\alpha}(\mathbb P|H_0+H_\infty,d)\to \VZ^{\leftrightarrow}_{1,\alpha}(\mathbb P_{\mathrm{log}},d).
\]
is log \'etale by the lifting criterion, as in~\cite[Theorem~5.3.4]{MW17}. The result is a consequence. 
\end{proof}

\subsection{Hypersurface pairs}\label{subsection virtual}
Let $(X,Y)$ be a smooth pair with $Y$ very ample. The definition in \S \ref{subsection factorisation} applies and yields a moduli space $\VZ_{1,\alpha}(X|Y,\beta)$ with a morphism
\begin{equation*} \VZ_{1,\alpha}(X|Y,\beta) \to \ol\Mcal^{\operatorname{log}}_{1,\alpha}(X|Y,\beta)\end{equation*}
obtained a logarithmic modification and the application of the factorisation property. This moduli space will typically be non-equidimensional, but we equip it with a virtual class as follows. The divisor $Y$ defines an embedding $X \hookrightarrow \PP^m$, with $Y=X\cap H$ for some hyperplane $H$.

\begin{lemma} The following square is cartesian in the category of ordinary stacks:
\bcd
\VZ_{1,\alpha}(X|Y,\beta) \ar[r] \ar[d] \ar[rd,phantom,"\square"] & \VZ_{1,\alpha}(\PP^m|H,d) \ar[d] \\
\VZ_{1,n}(X,\beta) \ar[r,"i"] & \VZ_{1,n}(\PP^m,d).
\ecd
\end{lemma}
\begin{proof} It is clear from the modular description that this square is cartesian in the category of fs logarithmic stacks, and since $i$ is strict, it is also cartesian in ordinary stacks.\end{proof}

Since $\VZ_{1,n}(\PP^m,d)$ is smooth and $\VZ_{1,n}(X,\beta)$ carries a virtual class \cite[Theorem 4.4.1]{RSPW}, the diagonal of $\VZ_{1,n}(\PP^m,d)$ defines a pullback~\cite[Appendix A]{BattistellaNabijou}, and therefore the virtual class on the space of maps to $(X,Y)$:
\begin{equation*} \virt{\VZ_{1,\alpha}(X|Y,\beta)} := i_\Delta^! [\VZ_{1,\alpha}(\PP^m|H,d)]. \end{equation*}
The virtual formalism is parallel to the one in~\cite{Ga}. As $\VZ_{1,\alpha}(X|Y,\beta)$ is equipped with evaluation maps and cotangent lines on the universal curve, we arrive at a definition of reduced Gromov--Witten invariants for the pair $(X,Y)$ by the standard mechanics. 

\section{Stratification and tropicalization}

\noindent It is know by Theorem~\ref{thm: log-smoothness} that $\VZ_{1,\alpha}(\mathbb P^m|H,d)$ is a toroidal orbifold. The irreducible components of its boundary stratify the space. The tropicalizations of logarithmic maps that factor completely are constrained by a condition called \textbf{well-spacedness}, c.f.~\cite[Section~4]{RSPW2}. 

\begin{definition}
Let $\plC$ be a tropical curve of genus one and let $\plE$ be its minimal subcurve of genus one. A tropical map $F: \plC\to \mathbb R_{\geq 0}$ is said to be \textbf{well-spaced} if one of the following two conditions are satisfied:
\begin{enumerate}
    \item no open neighbourhood of $\plE$ is contracted to a point in $\mathbb R_{>0}$, or
    \item\label{pt:well-spaced2} if an open neighbourhood of $\plE$ is contracted and $t_1,\ldots,t_k$ are the flags whose vertex is mapped to $F(\plE)$, but along which $F$ has nonzero slope; then, the minimum of the distances from $\plE$ to a vertex supporting $t_i$ occurs for at least two indices $i$.
\end{enumerate}
\end{definition}

\begin{figure}
 \begin{tikzpicture}
  %source curve
  \draw[gray] (0,-1) to[out=150, in=210, edge node={node [left] {$\plE$}}] (0,1);
  \draw[gray] (0,-1) to[out=30, in=-30] (0,1);
  \draw (0,1) --node[right]{$\ell_1$} (0,2) --node[above]{$t_1$} (-4,2) (0,2) --node[above]{$t_2$} (4,2);
  \draw[dotted] (-4.5,2) -- (-4,2) (4,2) -- (4.5,2);
  \draw (0,-1) --node[right]{$\ell_2$} (0,-2) --node[above]{$t_3$} (-4,-2) (0,-2) --node[above]{$t_4$} (4,-2);
  \draw[dotted] (-4.5,-2) -- (-4,-2) (4,-2) -- (4.5,-2);
  \foreach \x in {(0,2), (0,-2)}
  \draw[fill=black] \x circle (2pt);
  \foreach \x in {(0,1), (0,-1)}
  \draw[fill=gray] \x circle (2pt);
  %arrow
  \draw[->] (0,-2.5)--node[left]{$F$}(0,-3.5);
  %target line
  \draw (-4,-4)--(4,-4);
  \draw[dotted] (-4.5,-4) -- (-4,-4) (4,-4) -- (4.5,-4);
  \draw[fill=black] (0,-4) node[below]{$F(\plE)$} circle (2pt);
  
 \end{tikzpicture}
\caption{An example of well-spacedness: condition \eqref{pt:well-spaced2} becomes $\ell_1=\ell_2$.}
\label{fig:exa_well-spaced}
\end{figure}
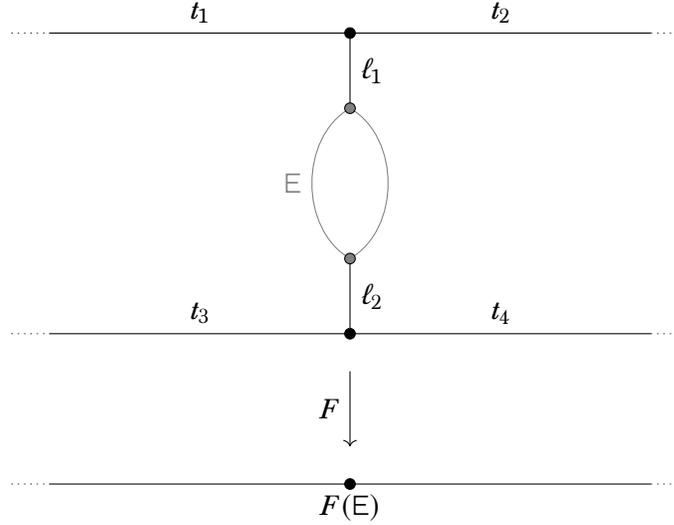

\begin{rem}
Condition \eqref{pt:well-spaced2} is not relevant to the rays of the tropical moduli space (i.e. divisors), since there are always at least two parameters involved, namely $F(\plE)$ - or the distance from $F(\plE)$ to the origin of $\RR_{\geq 0}$ - and the distance from $\plE$ to the first flag of non-zero slope within $\plC$. As we mostly concentrate on divisors starting from \S \ref{Section Gathmann line bundle}, it will not play much of a role in the rest of the paper.
\end{rem}

\begin{prop}\label{prop: well-spaced}
Let $[C\to \mathbb P^m]$ be a logarithmic stable map from a centrally aligned curve to an expansion, that factors completely. Then the tropicalization $\plC\to \mathbb R_{\geq 0}$ is well-spaced.
\end{prop}

\begin{proof}
The required result is a consequence of~\cite[Section 4]{RSPW2}, and we explain how to deduce it. We may focus on a single component of the expansion $\mathbb P^m[s]$ that contains the image of a contracted genus one subcurve, as this is the only relevant case. We replace the target with the projective bundle $\mathbb P(\mathcal O(1)\oplus \mathcal O)$ over $\mathbb P^{N-1}$ equipped with the divisorial logarithmic structure from the $0$ and $\infty$ sections. Let $p$ be the point to which the genus one subcurve is contracted. Passing to an open neighbourhood of $p$, the map to the bundle is given by a rational function $f$ on an open curve $C^\circ$, determined by the subgraph formed by $\plE$ and the flags $t_i$ described above. To describe the tropical map to $\mathbb R_{>0}$, we observe that $\plE$ is contracted to a fixed point $q\in\mathbb R_{>0}$. The flags at $t_i$ correspond to nodes or markings of $C$. The pole orders of $f$ at these distinguished points determine the slopes of the tropical map. We are now in the situation considered in~\cite[Second Paragraph of \S~4.6]{RSPW2}, and Lemma~4.6.1 of loc. cit. guarantees the well-spacedness.
\end{proof}

\subsection{The cone complex} To understand the stratification via combinatorial data, we will build the stratification from known objects. \smallskip

\noindent
\textbf{Step 1}. Let $\Sigma^{\mathrm{log}}$ be the tropical moduli space of genus one tropical stable maps to $\mathbb R_{\geq 0}$. This is naturally identified with the tropicalization (in the logarithmic sense) of the Abramovich--Chen--Gross--Siebert space of logarithmic stable maps to the pair $(\mathbb P^m,H)$. \smallskip

\noindent
\textbf{Step 2}. Given such a tropical stable map, we may subdivide $\mathbb R_{\geq 0}$ such that the image of every vertex of $\plC$ is a vertex of $\mathbb R_{\geq 0}$. Call this subdivision $\widetilde{\mathbb R}_{\geq 0}$. The preimages of vertices of the subdivision form a subdivision of $\plC$. After this procedure, the images of vertices of $\plC$ which lie in $\RR_{>0}$ are totally ordered, in a manner extending the partial order obtained from the map to $\mathbb R_{\geq 0}$. The combinatorial types of such \textbf{image-ordered} maps produce the cones of a subdivision of $\Sigma^{\mathrm{log}}$ which we denote $\Sigma^{\mathrm{Kim}}$. As the notation suggests, $\Sigma^{\Kim}$ is the tropicalization of Kim's space of logarithmic maps to $(\PP^m,H)$.\smallskip

\noindent
\textbf{Step 3.} Given a tropical map $F: \plC\to \mathbb R_{\geq 0}$ parameterised by $\Sigma^{\mathrm{Kim}}$, there is a largest radius $\delta$ (possibly equal to $0$) such that every vertex strictly contained in the circle of radius $\delta$ around the core has degree-marking $d=0$. Let $
\Sigma^{\mathrm{cen}}$ be the subdivision obtained by requiring that the vertices contained within the circle of radius $\delta$ around the core of $\plC$ are totally ordered. This involves introducing cones along which certain $\Z$-linear combinations of edge lengths are identified.\smallskip

\noindent
\textbf{Step 4}. Let $\Sigma_{1,\alpha}(\mathbb P^m|H,d)$ be the subcomplex of $\Sigma^{\mathrm{cen}}$ consisting of well-spaced tropical maps.

\begin{prop}
The cone complex $
\Sigma_{1,\alpha}(\mathbb P^m|H,d)$ is the fan of $
\VZ_{1,\alpha}(\mathbb P^m|H,d)$ viewed as a toroidal embedding. In particular, the codimension $k$ strata of $
\VZ_{1,\alpha}(\mathbb P^m|H,d)$ are in inclusion reversing bijection with the dimension $k$ cones in $
\Sigma_{1,\alpha}(\mathbb P^m|H,d)$.
\end{prop}

\begin{proof}
The construction above has been given to mimic the construction of the space $
\VZ_{1,\alpha}(\mathbb P^m|H,d)$. Specifically, the fact that $\Sigma^{\mathrm{cen}}$ is the cone complex attached to the logarithmic stack $\widetilde{\VZ}_{1,\alpha}(\mathbb P^m|H,d)$ of centrally aligned maps to expansions, follows from its description as a subcategory of the fibred category (over logarithmic schemes) of $\mathcal M^{\mathrm{log}}_{1,\alpha}(\mathbb P^m|H,d)$. To complete the result, note that $
\VZ_{1,\alpha}(\mathbb P^m|H,d)$ has a strict map to $
\VZ_{1,\alpha}(\mathbb P^m|H,d)$, so its cone complex is a subcomplex of $\Sigma^{\mathrm{cen}}$. It is contained in the subcomplex of well-spaced curves is by Proposition~\ref{prop: well-spaced}.
\end{proof}

\subsection{Indexing the strata} \label{subsection indexing strata} The dimension-$k$ cones in $\Sigma_{1,\alpha}(\PP^m|H,d)$ can be enumerated as follows. First, the cones in $\Sigma=\Sigma^{\Kim}$ are indexed by \textbf{combinatorial types} of tropical maps to $\RR_{\geq 0}$. Here a combinatorial type encodes all of the data of a tropical map, except for the edge lengths. To be more precise, a combinatorial type $\Delta$ consists of:
\begin{enumerate}
\item a finite graph $\sqC$ and topological (not metric) subdivision $\widetilde\RR_{\geq 0}$ of $\RR_{\geq 0}$;
\item genus, degree and marking assignments on the vertices;
\item the data of which vertex of $\widetilde\RR_{\geq 0}$ each vertex of $\sqC$ is mapped to;
\item integral slopes along the edges (both finite and infinite).
\end{enumerate}
The corresponding cone $\sigma \in \Sigma^{\operatorname{Kim}}$ is then given by the resulting moduli space of tropical maps, given by choices of edge lengths on source and target which produce a continuous tropical map. Given $\sigma\in \Sigma^{\operatorname{Kim}}$ we then produce all the cones in $\Sigma_{1,\alpha}(\PP^m|H,d)$ mapping to $\sigma$ by performing Steps 3--4 outlined above. This amounts to a subdivision of $\sigma$. In this process, new cones are created which map into larger-dimensional cones of $\Sigma^{\operatorname{Kim}}$. The process for enumerating the codimension $k$ strata of $\VZ_{1,\alpha}(\PP^m|H,d)$ is therefore:
\begin{enumerate}
\item fix a combinatorial type $\Delta$;
\item perform the subdivision of the resulting cone $\sigma \in \Sigma^{\operatorname{Kim}}$;
\item identify the dimension $k$ cones of that subdivision.
\end{enumerate}
In \S \ref{section reduced splitting} we perform this analysis in the case $k=1$.

\section{Degeneration of contact order}\label{Section Gathmann line bundle}
\noindent Consider the moduli space $\VZ_{1,\alpha}(\PP^m|H,d)$ of relative maps that factor completely.
\begin{dfn}
 For a marking $x_k$, consider the locus $\Dcal_{1,\alpha,k}(\PP^m|H,d) \subseteq \VZ_{1,\alpha}(\PP^m|H,d)$ where $x_k$ belongs to an \textbf{internal component} of the collapsed map, i.e. a component of the curve that is mapped into $H$.
\end{dfn}
%\noindent Equivalently, for the expanded map, $x_k$ is mapped to a higher level and does not belong to a \emph{trivial bubble} component.
\noindent In this section, we use the logarithmic structure on $\VZ_{1,\alpha}(\PP^m|H,d)$ to construct a line bundle $\Lcal_k$ and section $s_k$ which vanishes precisely along $\Dcal_{1,\alpha,k}(\PP^m|H,d)$. We use the correspondence with tropical geometry to identify $\cchern_1(\Lcal_k)$ in terms of tautological classes on $\VZ_{1,\alpha}(\PP^m|H,d)$, and to compute the vanishing order of $s_k$ along the components of $\Dcal_{1,\alpha,k}(\PP^m|H,d)$. Combined with the relative splitting formulas from the next section, this produces a recursion relation inside $\VZ_{1,\alpha}(\PP^m|H,d)$

The pair $(\Lcal_k,s_k)$ is natural on the moduli space $\MLog_{1,\alpha}(\PP^m|H,d)$ of non-expanded logarithmic maps; the corresponding pair on $\VZ_{1,\alpha}(\PP^m|H,d)$ will be obtained via pull-back. Consider therefore the tropicalization $\Sigma^{\log}$ of $\MLog_{1,\alpha}(\PP^m|H,d)$, identified as usual with the moduli space of tropical maps to $\RR_{\geq 0}$. We have a universal family
\bcd
\sqC \ar[d,"\pi"] \ar[r,"f"] & \RR_{\geq 0} \\
\Sigma^{\log} \ar[u, bend left=40pt, "x_k" left]
\ecd
where $x_k$ is the section which for every point $\lambda \in \Sigma^{\log}$ picks out the vertex of $\sqC_\lambda$ supporting the leg $x_k$. The composition $f \circ x_k \colon \Sigma^{\log} \to \RR_{\geq 0}$ defines a piecewise-linear function on $\Sigma^{\log}$ whose preimage over the open cone $\RR_{>0}$ consists of those tropical maps where $x_k$ belongs to an internal component. This produces a section of the ghost sheaf on $\MLog_{1,\alpha}(\PP^m|H,d)$, which in the usual way induces a line bundle and section $(\Lcal_k,s_k)$ on the moduli space, and the tropical description above shows that the zero locus of $s_k$ is (set-theoretically) the locus where $x_k$ belongs to an internal component.\medskip

We now calculate $\cchern_1(\Lcal_k)$. 
\begin{thm} $\cchern_1(\Lcal_k) = \alpha_k \psi_k + \ev_k^\st H.$\end{thm}
\begin{proof}
Choose a family of logarithmic stable maps over $S$ and let $\mu \in \Gamma(S,\ol{M}_S)$ be the global section of the ghost sheaf constructed in the previous paragraph. This pulls back along $\pi$ to give a global section $\pi^\flat(\mu) \in \Gamma(C,\ol{M}_C)$.  Interpreted as a piecewise-linear function on the tropicalization $\sqC$ with values in $\ol{M}_S$ \cite[Remark 7.3]{CavalieriChanUlirschWise}, this assigns $\mu$ to every vertex and has slope zero along every edge. By construction, the line bundle associated to this section is $\pi^\st \Lcal_k$. Consider on the other hand the generator $1 \in \N = \Gamma(\PP^m,\ol{M}_{\PP^m})$ with associated line bundle $\OO(H)$. The section $f^\flat(1) \in \Gamma(C,\ol{M}_C)$ has associated line bundle $f^\st\OO(H)$. If we let $v$ denote the vertex containing $x_k$, then by construction $f^\flat(1)$ assigns $\mu$ to $v$ and has slope $\alpha_k$ along the leg $x_k$. Thus if we consider the difference $f^\flat(1) - \pi^\flat(\mu)$ then this assigns $0$ to $v$ and still has slope $\alpha_k$ along $x_k$. Thus by \cite[Proposition 2.4.1]{RSPW} the corresponding line bundle restricted to $C_v$ is given by
\begin{equation*} \OO_{C_v} \left(\alpha_k x_k + \sum_e \mu_e x_e \right) \end{equation*}
where the sum is over the edges $e$ adjacent to $v$ and distinct from $x_k$. Thus we have:
\begin{equation*} \left( f^\st\OO(H) \otimes \pi^\st \Lcal_k^{-1} \right) \big|_{C_v} = \OO_{C_v} \left(\alpha_k x_k + \sum_e \mu_e x_e \right).\end{equation*}
Since $x_k$ factors through $C_v$ we may pull back along $x_k$ to obtain
\begin{equation*} \Lcal_k = x_k^\st\pi^\st \Lcal_k = x_k^\st \OO_{C_k}(-\alpha_k x_k) \otimes x_k^\st f^\st\OO(H) = x_k^\st \OO_{C_k}(-\alpha_k x_k) \otimes \ev_k^\st \OO(H) \end{equation*}
and taking Chern classes gives
\begin{equation*} \cchern_1(\Lcal_k) = \alpha_k \psi_k + \ev_k^\st H,\end{equation*}
as required.\end{proof}
This produces $(\Lcal_k,s_k)$ and calculates $\cchern_1(\Lcal_k)$ on $\MLog_{1,\alpha}(\PP^m|H,d)$; this bundle is pulled back to $\VZ_{1,\alpha}(\PP^m|H,d)$. The relevant piecewise-linear function is the composition:
\begin{equation*} \Sigma_{1,\alpha}(\PP^m|H,d) \to \Sigma^{\log} \to \RR_{\geq 0}.\end{equation*}
Note in particular that $\psi_k$ should be interpreted as a \textbf{collapsed class} on $\VZ_{1,\alpha}(\PP^m|H,d)$, i.e. the cotangent class from the source curve of the collapsed and stabilised map. We arrive at the recursion. 

\begin{thm} \label{theorem recursion} We have the following relation in the Chow ring of $\VZ_{1,\alpha}(\PP^m|H,d)$:
\begin{equation}\label{equation recursion} (\alpha_k \psi_k + \ev_k^\st H) \cap [\VZ_{1,\alpha}(\PP^m|H,d)] = \sum_{\Dcal} \lambda_\Dcal [ \Dcal ].\end{equation}
The sum is over irreducible components of the divisor $\Dcal_{1,\alpha,k}(\PP^m|H,d) \subseteq \VZ_{1,\alpha}(\PP^m|H,d)$, and $\lambda_\Dcal$ is the vanishing order of $s_k$ along this component. \end{thm}

\begin{remark} This construction gives the logarithmic analogue of Gathmann's line bundle and section \cite[Construction 2.1]{Ga}. The logarithmic approach makes the computation of vanishing orders combinatorial (see \S \ref{subsubsection splitting degree} below), circumventing a difficult calculation given by Gathmann. \end{remark}

\section{Splitting the boundary}\label{section reduced splitting}

\noindent A basic phenomenon is responsible for the nature of the forthcoming analysis. In contrast with nodal singularities, the functions on an elliptic singularity are not simply collections of functions on the normalisation that agree at the gluing point. Rather, the contraction map $C\to\overline C$ corresponds to a particular linear combination of tangent vectors at the pointed normalisation, whose kernel consists of those functions descending to the elliptic singularity (see \S \ref{S:ellsing}). In order to compute integrals, it is necessary to give a tautological description of this additional condition. % that captures which functions (or maps) descend from the normalization to the singularity. 

We provide an explicit description of the terms appearing in the right-hand side of \eqref{equation recursion}.  For each term, we provide a combinatorial formula for the vanishing order $\lambda_\Dcal$, and a recursive description of the stratum $\Dcal$ 
in terms of fibre products of moduli spaces with smaller numerical data. The tautological description of the factorisation condition then leads to a recursive structure for calculating integrals. This step is the heart of the analysis. 

\begin{remark} For the reduced relative Gromov--Witten theory of $(\mathbb P^m,H)$ with only primary insertions, i.e. where all cohomological insertions are pulled back along evaluation maps, computations can be carried out without an analysis of the factorisation condition; this follows from dimension considerations and leads to the main results in~\cite{Vre}. However, to include descendant insertions, work with more general pairs, or produce cycle-level statements, an analysis of the tautological classes arising from the factorisation condition is crucial.
\end{remark}

\subsection{Irreducible components of the degenerate locus} We explain the tropical procedure to identify irreducible components of $\Dcal_{1,\alpha,k}(\PP^m|H,d)$.

\begin{lemma} \label{Lemma components are logarithmic divisors} Every irreducible component of $\Dcal_{1,\alpha,k}(\PP^m|H,d)$ is a codimension $1$ stratum.\end{lemma}

\begin{proof} The locus where the logarithmic structure is trivial is the locus where the source curve is smooth and not mapped inside $H$. Since $\VZ_{1,\alpha}(\PP^m|H,d)$ is toroidal this locus is open and dense. By definition $\Dcal_{1,\alpha,k}(\PP^m|H,d)$ is contained in the  complement of this locus and can be expressed as a union of strata of positive codimension. Since $\Dcal_{1,\alpha,k}(\PP^m|H,d)$ is pure codimension $1$, it is a union of logarithmic divisors.\end{proof}

\noindent In \S \ref{subsection indexing strata} we described the logarithmic strata of $\VZ_{1,\alpha}(\PP^m|H,d)$, using the cones of the tropicalization $\Sigma_{1,\alpha}(\PP^m|H,d)$. Every divisorial logarithmic stratum is obtained as follows:
\begin{enumerate}
\item choose a combinatorial type $\Delta$ of a tropical map;
\item subdivide the corresponding tropical moduli space $\sigma$ to align the type;
\item choose a ray in this subdivision.
\end{enumerate}
This process contains redundancies: upon choosing a ray in the tropical moduli space, some of the edge lengths or vertex positions may get set to $0$. This induces a generisation of the initial combinatorial type $\Delta$, given by contracting the corresponding edges of the dual graph and moving the corresponding vertices from $\RR_{>0}$ to $0$. When we speak of the combinatorial type of a stratum in $\VZ_{1,\alpha}(\PP^m|H,d)$, we will always mean this generisation. This is independent of the choice of $\Delta$, and in fact we can and will always choose $\Delta$ to coincide with the generisation.

The logarithmic divisors contained in $\Dcal_{1,\alpha,k}(\PP^m|H,d)$ are those whose  combinatorial types map the vertex of the dual graph containing $x_k$  into the interior $\RR_{>0} \subseteq \RR_{\geq 0}$. Thus via the above procedure, we enumerate the components of $\Dcal_{1,\alpha,k}(\PP^m|H,d)$ in a combinatorial manner; the combinatorial type $\Delta$ allows us to describe the general element of such a component. The remainder of this section describes the components and a recursive computation of integrals over them.

\subsection{Recursive description of the divisors: types I, II and III} Choose an irreducible component $\Dcal \subseteq \Dcal_{1,\alpha,k}(\PP^m|H,d)$. By Lemma \ref{Lemma components are logarithmic divisors} this is a logarithmic divisor, and hence may be written as:
\begin{equation*} \Dcal = \widetilde\Dcal \cap \VZ_{1,\alpha}(\PP^m|H,d) \end{equation*}
for a unique logarithmic stratum $\widetilde\Dcal \subseteq \widetilde{\VZ}_{1,\alpha}(\PP^m|H,d)$. Since
\begin{equation*}\widetilde{\VZ}_{1,\alpha}(\PP^m|H,d) \to \ol\Mcal_{1,\alpha}(\PP^m|H,d)\end{equation*}
is a logarithmic modification, the divisor $\widetilde{\Dcal}$ is either exceptional or the strict transform of a logarithmic divisor. By the nature of the subdivision procedure these two cases correspond, respectively, to when the core is assigned zero degree or nonzero degree by the combinatorial type.

We begin with the latter. Suppose therefore that $\widetilde{\Dcal}$ is the strict transform of a logarithmic divisor:
\begin{equation*} \Ecal \subseteq \ol\Mcal_{1,\alpha}(\PP^m|H,d). \end{equation*}
The birational map $\widetilde{\Dcal} \to \Ecal$ induces a morphism $\Dcal \to \Ecal$. We now show how to interpret this morphism as a desingularisation of the principal component. Since $\Ecal$ admits a recursive description in terms of relative and rubber moduli spaces, this will allow us to compute integrals over $\Dcal$.

\begin{lemma}\label{lem:combs} Let $\Dcal \subseteq \Dcal_{1,\alpha,k}(\PP^m|H,d)$ be an irreducible component which contributes nontrivially to the Gromov--Witten invariant, and let $\Delta$ be the corresponding combinatorial type. If $\Delta$ assigns positive degree to the core, then $\Delta$ takes one of the following forms:
\begin{figure}[h]
    \centering
    \begin{minipage}{0.3\textwidth}
        \centering
        \Yagraph
        \caption{I}
    \end{minipage}\hfill
    \begin{minipage}{0.3\textwidth}
        \centering
        \Ybgraph
        \caption{II}
    \end{minipage}\hfill
    \begin{minipage}{0.3\textwidth}
        \centering
        \Ycgraph
        \caption{III}
    \end{minipage}
   
\end{figure}

\noindent where black vertices represent rational components, white vertices components of genus one. The degrees of vertices, expansions factors, and remaining markings are distributed arbitrarily, subject to the following constraints:
\begin{enumerate}
\item the vertices $\sqC_1,\ldots,\sqC_r$ have positive degree;
\item the core has positive degree;
\item every vertex is stable;
\item the balancing condition is satisfied.
\end{enumerate}\end{lemma}
\noindent In \cite{Vre} these cases are referred to, respectively, as types $A,B$ and $C^+$.

\begin{proof} Let $\sigma \in \Sigma_{1,\alpha}(\PP^m|H,d)$ be the cone corresponding to $\Ecal \subseteq \ol\Mcal_{1,\alpha}(\PP^m|H,d)$. By the discussion above $\Dcal$ corresponds to a ray in the subdivision of $\sigma$ obtained by imposing the central alignment condition. Since the elliptic core is assigned positive degree, both radii are equal to $0$ and the subdivision is trivial. We conclude that $\sigma=\RR_{\geq 0}$. Since there must be at least one vertex mapped to higher level, the tropical target $\widetilde\RR_{\geq 0}$ is obtained from $\RR_{\geq 0}$ by subdividing at a single point $\diamond \in \RR_{> 0}$.

In order for the cone $\sigma$ to be $1$-dimensional, the dual graph $\Gamma$ must be bipartite, with vertices over $0$ and $\diamond$. The cases I, II and III enumerated above cover situations where there is a single vertex mapped to $\diamond$. If there is more than one such vertex, then the contribution to the Gromov--Witten invariant vanishes. To see this, we consider the stratum
\begin{equation*} \Ecal^{\text{\tiny{log}}} \subseteq \ol\Mcal^{\text{\tiny{log}}}_{1,\alpha}(\PP^m|H,d) \end{equation*}
to which $\Ecal$ maps under the collapsing morphism, and examine the composition $\Dcal \to \Ecal \to \Ecal^{\text{\tiny{log}}}$. If we let $\Fcal^{\text{\tiny{log}}}$ denote the intersection of $\Ecal^{\text{\tiny{log}}}$ with the main component of the moduli space, then we obtain a factorisation $\Dcal \to \Fcal^{\text{\tiny{log}}} \hookrightarrow \Ecal^{\text{\tiny{log}}}$. Since the moduli space is generically unobstructed along $\Fcal^{\text{\tiny{log}}}$, the codimension of $\Fcal^{\text{\tiny{log}}}$ is given by the dimension of the associated cone in the tropicalization $\Sigma^{\text{\tiny{log}}}$. If there is more than one vertex mapped to higher level, then this cone has dimension $\geq 2$. Therefore the map $\Dcal \to \Fcal^{\text{\tiny{log}}}$ has positive-dimensional fibres, and since all insertions are pulled back from the latter space the contribution vanishes by the projection formula.\end{proof}

\begin{remark} The difference in dimensions between $\Dcal$ and $\Fcal^{\text{\tiny{log}}}$ (or equivalently, between virtual dimensions of $\Ecal$ and $\Ecal^{\text{\tiny{log}}}$) may be interpreted as the difference in dimensions between moduli spaces of \emph{disconnected} rubber and their images under the collapsing morphisms. \end{remark}

We investigate the three types I, II and III separately, giving a recursive description of the boundary divisor in each case. For the remainder of this subsection, therefore, we fix a one-dimensional cone $\tau \in \Sigma_{1,\alpha}(\PP^m|H,d)$, let $\Dcal \subseteq \VZ_{1,\alpha}(\PP^m|H,d)$ be the associated logarithmic divisor, and assume that $\Dcal$ is contained in $\Dcal_{1,\alpha,k}(\PP^m|H,d)$ and is of type I, II or III. We let $\Ecal \subseteq \ol\Mcal_{1,\alpha}(\PP^m|H,d)$ be the logarithmic stratum into which $\Dcal$ is mapped; this is indexed by a cone $\epsilon \in \Sigma^{\Kim}$ corresponding to a combinatorial type $\Delta$, and $\epsilon$ is one-dimensional since we are restricting to the type I, II, III cases.

\subsubsection{Type I}\label{subsubsection type A} Suppose $\Delta$ is of type I. Then $\Ecal$ admits a finite and surjective \textbf{splitting morphism} onto the fibre product:
\begin{equation*} \Ecal \to \left( \ol\Mcal_{1,\alpha^{(1)}\cup(m_1)}(\PP^m|H,d_1) \times \prod_{i=2}^r \ol\Mcal_{0,\alpha^{(i)}\cup(m_i)}(\PP^m|H,d_i) \right) \times_{H^r} \ol\Mcal^{\leftrightarrow}_{0,\alpha^{(0)}\cup (-m_1,\ldots,-m_r)}(\mathbb{P}|H_0+H_\infty,d_0).\end{equation*}

\begin{lemma} \label{Lemma type A gluing} The divisor $\Dcal$ admits a natural splitting morphism
\begin{equation*}\Dcal \xrightarrow{\rho} \left(\VZ_{1,\alpha^{(1)}\cup(m_1)}(\PP^m|H,d_1)\times\prod_{i=2}^r \ol\Mcal_{0,\alpha^{(i)}\cup(m_i)}(\PP^m|H,d_i)\right) \times_{H^r} \ol\Mcal^{\leftrightarrow}_{0,\alpha^{(0)}\cup (-m_1,\ldots,-m_r)}(\mathbb{P}|H_0+H_\infty,d_0)\end{equation*}
such that the map $\Dcal \to \Ecal$ lifts the map on fibre products obtained by desingularising the main component of the first factor.\end{lemma}

\begin{proof}
We consider the statement of the lemma in logarithmic schemes, so that a relative stable map to an expansion admits a unique logarithmic lifting, and later saturate the structures. The fibre product description preceding the lemma is an isomorphism in this category~\cite[Lemma~4.2.2]{AbramovichMarcusWiseComparison}.

The map $\Dcal\to\Ecal$ is an isomorphism away from the exceptional centres, and the latter is contained in the locus where the core is contracted. Given an element of $\Dcal$ we can split it along the nodes $q_1,\ldots,q_r$. It is then clear that $\sqC_1$ is aligned. We claim that $\sqC$ satisfies the factorisation property if and only if $\sqC_1$ does. This implies the lemma.

On $\sqC$ there are two contraction radii $\delta_F,\delta_B$. Let $\delta \in \{\delta_F,\delta_B\}$. An examination of the stratification by combinatorial type of $\Dcal$ shows that $\lambda(\sqC^\prime) > \delta$ for any component $\sqC^\prime$ of $\sqC_0$. This means that $\sqC_0,\sqC_2,\sqC_3,\ldots,\sqC_r$ lie outside both contraction radii, and so the aligned curve $\sqC$ satisfies the factorisation condition if and only if $\sqC_1$ does. \end{proof}

%\begin{lemma}\label{type A radius lemma} In the notation of the proof above, $\lambda(\sqC^\prime) > \delta$ for any component $\sqC^\prime$ of $\sqC_0$.\end{lemma}
%
%\begin{proof} Observe first that $\Dcal$ is obtained as the disjoint union of the locally closed logarithmic strata contained in the closure of the stratum where all of the $\sqC_i$ are irreducible. We choose and examine one of these boundary strataIf we look at one of these boundary strata, then by definition the associated cone $\sigma \in \Sigma_{1,\alpha}(\PP^m|H,d)$ contains a ray $\tau$ corresponding to the stratum where all the $\sqC_i$ are irreducible; this amounts to setting all edge lengths other than $e_1,\ldots,e_r$ to zero. If $f_1,\ldots,f_l$ are some collection of additional edge lengths in $\sigma$ (corresponding to internal nodes in degenerations of the $\sqC_i$) then since $\sigma$ is adjacent to $\tau$ we must have (by construction of the subdivision) $f_1+\ldots+f_l < e_j$ for all $j\in\{1,\ldots,r\}$, since $f_1=\ldots=f_l=0$ and $e_j \neq 0$ on $\tau$. In particular, if $\sqC_1$ is degenerate and if $\delta$ denotes the minimal distance from the core to a non-contracted vertex of $\sqC_1$ (which certainly exists since $\sqC_1$ has positive degree) then $\delta < e_1 \leq \lambda(\sqC^\prime)$, as claimed.\end{proof}
\noindent Lemma \ref{Lemma type A gluing} provides a means to calculate integrals over the class $\lambda_\Dcal[\Dcal]$ appearing on the right-hand side of Theorem \ref{theorem recursion}, provided that we can calculate $\lambda_\Dcal$ and the degree of the splitting morphism. Simple closed formulae for these are given in \S \ref{subsubsection splitting degree} below.

\subsubsection{Type II}
Now suppose $\Delta$ is of type II. In this case, it is impossible for the core to be contracted. Hence $\Ecal$ is disjoint from the blown-up locus, and the map $\Dcal \to \Ecal$ is an isomorphism. We obtain the following description, entirely in terms of genus zero data:
\begin{lemma} $\Dcal$ admits a finite and surjective splitting morphism:
\begin{equation*} \Dcal \xrightarrow{\rho} \left(\ol\Mcal_{0,\alpha^{(1)}\cup(m_1,m_2)}(\PP^m|H,d_1)\times\prod_{i=3}^r \ol\Mcal_{0,\alpha^{(i)}\cup(m_i)}(\PP^m|H,d_i)\right) \times_{H^r} \ol\Mcal^{\leftrightarrow}_{0,\alpha^{(0)}\cup(-m_1,\ldots,-m_r)}(\mathbb P|H_0+H_\infty,d_0).\end{equation*}\end{lemma}

\subsubsection{Type III} \label{subsubsection type C+} Finally suppose that $\Delta$ is of type III. As before we have a finite and surjective morphism:
\begin{equation*} \Ecal \to  \left( \prod_{i=1}^r \ol\Mcal_{0,\alpha^{(i)}\cup(m_i)}(\PP^m|H,d_i) \right) \times_{H^r} \ol\Mcal^{\leftrightarrow}_{1,\alpha^{(0)}\cup(-m_1,\ldots,-m_r)}(\mathbb{P}|H_0+H_\infty,d_0). \end{equation*}
The same arguments as in \S \ref{subsubsection type A} then apply to give:
\begin{lemma} $\Dcal$ admits a finite splitting morphism
\begin{equation*}\Dcal \xrightarrow{\rho} \left( \prod_{i=1}^r \ol\Mcal_{0,\alpha^{(i)}\cup(m_i)}(\PP^m|H,d_i) \right) \times_{H^r} \VZ^{\leftrightarrow}_{1,\alpha^{(0)}\cup(-m_1,\ldots,-m_r)}(\mathbb{P}|H_0+H_\infty,d_0)\end{equation*}
such that $\Dcal \to \Ecal$ corresponds to the obvious map on fibre products.\end{lemma}
\noindent The final factor in the fibre product is the logarithmic blowup of the moduli space of rubber maps constructed in \S \ref{subsection rubber}. We will compute integrals over such spaces as part of the recursion (see \S \ref{section recursion algorithm}).

\subsubsection{Splitting degree and vanishing order} \label{subsubsection splitting degree} In each of the subsections \ref{subsubsection type A}--\ref{subsubsection type C+} above, we obtained a finite splitting morphism $\rho$ from $\Dcal$ to a fibre product of moduli spaces with smaller numerical data. Here we describe the degree of $\rho$ and calculate the vanishing order $\lambda_\Dcal$ of the section $s_k$ constructed in \S \ref{Section Gathmann line bundle}. This completely describes the terms appearing on the right-hand side of Theorem \ref{theorem recursion} which are of type I, II or III.

\begin{lemma}\label{lem:saturation} The degree of $\rho$ is given by:
\begin{equation*} \label{degree of gluing} \dfrac{\prod_{i=1}^r m_i}{\lcm(m_1,\ldots,m_r)}. \end{equation*}\end{lemma}
\begin{proof} This calculation is well known, see for instance~\cite[Section~7.9]{ChenDegeneration}, \cite[Section~5.3]{ACGSDecomposition}.
\end{proof}

\begin{lemma}\label{lemma vanishing order} The section $s_k$ vanishes along the divisor $\Dcal$ with order given by 
\[
\lambda_\Dcal = \lcm(m_1,\ldots,m_r).
\] 
\end{lemma}
\begin{proof} Let $\tau \in \Sigma_{1,\alpha}(\PP^m|H,d)$ be the cone corresponding to $\Dcal$ and let
\begin{equation*} \varphi \colon \Sigma_{1,\alpha}(\PP^m|H,d) \to \Sigma^{\log} \to \RR_{\geq 0} \end{equation*}
be the piecewise-linear function constructed in \S \ref{Section Gathmann line bundle}. It follows from the tropical description that $\lambda_\Dcal$ is equal to the index of the map of integral cones $\tau \to \RR_{\geq 0}$ obtained by restricting $\varphi$. Observe that $\tau \subseteq \RR_{\geq 0}^r$ with integral generator:
\begin{equation*} w = \left( \dfrac{\lcm(m_1,\ldots,m_r)}{m_1},\ldots,\dfrac{\lcm(m_1,\ldots,m_r)}{m_r} \right).\end{equation*}
The map $\tau \to \RR_{\geq 0}$ is given by projecting onto the $i$th factor and then multiplying by $m_i$. Note that by piecewise linear continuity of the tropical map, this is independent of $i$. We conclude that the index is equal to $\lcm(m_1,\ldots,m_r)$ as claimed.\end{proof}

In summary, the contribution of each term $\lambda_\Dcal [\Dcal]$ is given by integrating over the appropriate fibre product and multiplying the result by $\lambda_\Dcal \cdot \deg\rho = \prod_{i=1}^r m_i$.

\subsection{Recursive description of the divisors: type $\dag$}\label{subsection C0 splitting} The treatment of the type I, II and III strata are direct extensions of ideas of  Gathmann--Vakil. The type $\dag$ case is entirely new. We provide a recursive description of such boundary divisors; this forms the technical heart of the paper.

\subsubsection{Possible combinatorial types}\label{S:combinatorialdescription} A boundary divisor $\Dcal \subseteq \Dcal_{1,\alpha,k}(\PP^m|H,d)$ is said to have \textbf{type $\dag$} if the core is assigned degree zero. Divisors of this type occur as principal components in compactified torus bundles over strata in the space of relative maps. As a result, their combinatorial types exhibit degenerate behaviour. Precisely, over the generic point of such a divisor, the target may expand multiple times, and the source tropical curve need not be bipartite with a single interior vertex. An instance of this can be seen the left hand side of Figure \ref{fig:off_we_go}: at first sight, the multiple expansions seen in the target appear to depict a high codimension stratum, but due to the alignment condition, which is required to pick out the main component, these lengths cannot be varied independently. The degenerate contributions are at the heart of the reduced theory and cannot be removed.

The next two lemmas describe the combinatorial types associated to rays whose divisors $\Dcal$ have type $\dag$ and contribute nontrivially to the Gromov--Witten invariants.
\begin{lemma} \label{lemma type C0 combinatorial types}
 For every nonzero vertex $\diamond$ of the tropical target, the fibre of the tropical map over $\diamond$ contains exactly one stable vertex, which either is the core, or lies on the circle.
\end{lemma}

\begin{lemma}\label{lem:stable_vertices} All the vertices on the circle are stable i.e. they are not Galois covers of fibers of a component of an expansion.
\end{lemma}

%\begin{lemma}\label{lem:corecontact}
% Among all the flags supported at the core, at least one must have expansion factor strictly greater than $1$.
%\end{lemma}

\begin{proof}[Proof of Lemma \ref{lemma type C0 combinatorial types}]
Stability of maps to an expansion requires that at least one stable vertex lies over each vertex of the target. If there is more than one stable vertex, then we show as in the proof of Lemma \ref{lem:combs} that the corresponding locus $\Fcal$ in the moduli space of unexpanded maps has high codimension, i.e. the collapsing map has fibres of positive dimension. As both insertions and the factorisation property for the collapsed map to $\mathbb P^m$ are pulled back from the moduli space of maps without expansions, the contribution vanishes by the projection formula.
\end{proof}
\begin{proof}[Proof of Lemma \ref{lem:stable_vertices}]
Since the corresponding family of tropical maps is one-dimensional, for every vertex of the target there must be at least one vertex above it which lies on the circle. If the vertex on the circle is strictly semistable -- so that in particular the collapsed map $f_B$ is constant along the corresponding component -- the factorisation condition for $f_B$ is unaffected by $\Gm$-scaling the coordinate corresponding to that vertex in the moduli space of attaching data \cite[\S 2.2]{SMY2}, so the map $\Dcal\to\Ecal$ has generic fibre of positive dimension and the contribution vanishes.
\end{proof}
\begin{comment}

\begin{proof}[Proof of Lemma \ref{lem:corecontact}]
Consider the vertices that lie above $0\in\mathbb R_{\geq0}$ on the circle of radius $\delta$: the edges departing from them go directly (up to bubbling) to the core, since there cannot be stable vertices inside the open disc  when we have only one tropical parameter. In particular, all these edges have the same expansion factor $m$ (see again Figure \ref{fig:off_we_go}), which is maximal among those of the edges adjacent to the core. Indeed, all the edges from the core to the circle have the same length $\delta$, so the ones having maximal contact order are those that reach further from the core, and at least one must reach $0\in\mathbb R_{\geq0}$, for otherwise there would be an extra tropical parameter.

Finally, in case the maximal contact order is $m=1$, in order for $f_B$ to factor, the image of the tangent vectors at the components lying on the circle must be linearly dependent in $T_{\PP^m,x}$, where $x$ denotes the image of the core under $f_B$.in which case we can deduce as above that factorisation will be satisfied independently of the choice of alignment.
\end{proof}

For the remainder of this section we use the following notation for the combinatorial type $\Delta$; we let $\sqC_0$ denotes the vertex of dual graph corresponding to the core, $\sqC_1,\ldots,\sqC_r$ the other stable vertices (genus zero and lying on the circle, by the previous lemma) and $q_1,\ldots,q_r$ the corresponding splitting nodes.
 
\end{comment}
\subsubsection{Lines}\label{S:lines}
We begin with a description of the boundary divisor corresponding to a comb whose teeth all have degree $1$. This is a concrete but atypical case, since the teeth are generically transverse (not just dimensionally) to the hyperplane $H$. For simplicity, we assume that the contact order is $(d)$, concentrated at a marking $x_1$ supported on the elliptic curve $E$ contracted to $H$. The combinatorial type is as in Figure \ref{fig:comb111}.

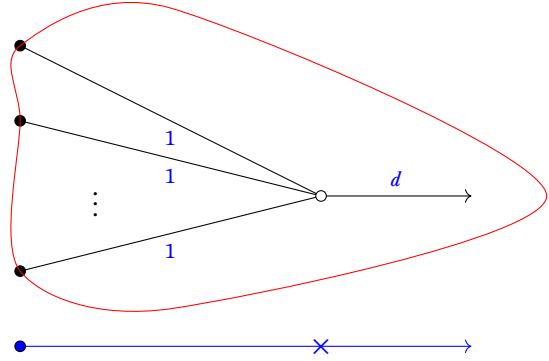
\begin{figure}
 \begin{tikzpicture}
 \tikzset{cross/.style={cross out, draw=black, fill=none, minimum size=2*(#1-\pgflinewidth), inner sep=0pt, outer sep=0pt}, cross/.default={3pt}}
  \foreach \x in {(0,2),(0,1),(0,-1)}
  \draw[fill=black] \x circle (2pt);
  \draw[fill=blue] (0,-2) circle (2pt);
  \foreach \x in {(0,2),(0,1),(0,-1)}
  \draw \x --node[below,blue]{\SMALL$1$} (4,0);
  \draw (1,0) node{$\vdots$};
  \draw[->] (4,0) --node[above]{\color{blue} \SMALL$d$} (6,0);
  \draw[blue,->] (0,-2)--(6,-2);
  \draw[fill=white] (4,0) circle (2pt);
  \draw[fill=blue,blue] (4,-2) node[blue]{$\times$};
  \draw [red] plot [smooth cycle] coordinates {(0,2) (0,1) (0,-1) (2,-1.5) (7,0) (2,2.5)};
 \end{tikzpicture}
 \caption{The combinatorial type of a comb with degree $1$ teeth.}
 \label{fig:comb111}
\end{figure}
Start with the moduli space of logarithmic stable maps. There is a finite splitting morphism:
\[\Ecal \to \left(\M{0}{(1)}{\PP^m|H}{1}^{\times m}\times_{H^m}H\right)\times \overline{\Mcal}^{\leftrightarrow}_{1,(-1,\ldots,-1,d)}(\PP^1|0+\infty,d),\]
where the latter is the moduli space of genus one rubber maps. To obtain $\widetilde{\Dcal}$ from $\Ecal$ we replace the final factor by $\VZ^{\leftrightarrow}_{1,(-1,\ldots,-1,d)}(\PP^1|0+\infty,d)$ (see \S \ref{subsection rubber}).

The first  factor can be described very concretely as the choice of a point $p$ of $H$, and $d$ lines through it. An important observation is that, generically, the vertices which lie on the circle of radius $\delta$ and map to $0\in\mathbb R_{\geq 0}$ are automatically aligned. Indeed, if $R_i$ denotes the $i$th tail of the curve, glued to $E$ at the node $q_i$, then the stable map provides us with identifications:
\[\operatorname{d}\! f_{i,q_i}\colon T_{R_i,q_i}\cong N_{H/\PP^m,p}\]
which we may compose to obtain attaching data $\theta_{ij}$  (see \S \ref{S:ellsing}). This produces a contraction $C\to\overline C$ to an elliptic $d$-fold point. The dimension of $\widetilde\Dcal$ coincides with the dimension of $\Ecal$, which is $dm + m - 1$.

Assume first that $d\leq m$. For the absolute factorisation to hold, it is necessary (but not sufficient) for the $d$ lines $f_i(R_i)$ to span a subspace of $T_{\PP^m,p}$ of dimension at most $d-1$. The locus of maps satisfying this condition has dimension:
\begin{equation*} (dm+m-1) - (m-d+1) = dm + d -2.\end{equation*}
Without loss of generality, we may assume that $H=\{z_0=0\}$, $p=[0,\ldots,0,1]$, and that the $d$ lines span the subspace $\{z_{d-1}=\ldots=z_{m-1}=0\}$. The sections $z_i$ for $i=d-1,\ldots,m$ vanish identically near $E$, so descend automatically to the singularity. On the other hand the section $z_0$ descends to the singularity by the construction of $\overline C$. Finally, the sections $z_i$ for $i=1,\ldots,d-2$ impose independent conditions to descend to $\overline C$. Imposing these conditions, we find that $\Dcal$ has dimension 
\begin{equation*} (dm+d-2)-(d-2) = dm \end{equation*}
making it into a divisor in $\VZ_{1,(d)}(\PP^m|H,d)$. The case $d>m$ is similar, except that the lines $f_i(R_i)$ are automatically linearly dependent and so this condition does not need to be imposed.

We express the factorisation condition tautologically by combining the isomorphisms $\theta_{ij}$ above with the differential of $f$. Factorisation amounts to the vanishing of the following vector bundle map:
\[T_{R_1,q_1}\xrightarrow{(1,\theta_{12}\ldots,\theta_{1d})}\bigoplus_{i=1}^d T_{R_i,q_i}\xrightarrow{\sum\circ\operatorname{d}\!f=\operatorname{d}\!z_1+\ldots\operatorname{d}\!z_{d-2}} {N_{H\cap\PP^{d-1}|\PP^{d-1}}}|_p.\]
For strata with higher tangency, we will see that the automatic alignment only occurs for vertices on the circle which map to $0\in\RR_{\geq 0}$. For internal components which lie on the circle, the alignment must be imposed separately. This means that $\widetilde\Dcal$ is a compactified torus bundle over $\Ecal$ (see \S\ref{S:ellsing}). On the other hand, for higher contact orders, the vector bundle map $\Sigma\circ\operatorname{d}\!f$ to $T_{\PP^m,p}$ factors naturally through $T_{H,p}$. This is discussed in detail in \S\S \ref{subsection D from Dtilde} and \ref{section D from Dtilde second}.

\subsubsection{Recursive description of $\Ecal$} Recall that we have maps
\[
\VZ_{1,\alpha}(\PP^m|H,d)\hookrightarrow \widetilde{\VZ}_{1,\alpha}(\PP^m|H,d)\to \overline\Mcal_{1,\alpha}(\PP^m|H,d).
\]
The first morphism is a closed embedding cut out by the two factorisation properties. The second morphism is a logarithmic modification. Let $\tau \in \Sigma_{1,\alpha}(\PP^m|H,d)$ be the ray corresponding to $\Dcal \subseteq \VZ_{1,\alpha}(\PP^m|H,d)$, and to the virtual divisor $\widetilde \Dcal \subseteq \widetilde{\VZ}_{1,\alpha}(\PP^m|H,d)$. 

Let $\sigma \in \Sigma^{\Kim}$ be the minimal cone containing the image of $\tau$. The dimension of this cone could be large (see again the left-hand side of Figure \ref{fig:off_we_go}) -- its dimension is equal to the number of levels of the target expansion in the type $\Delta$. The cone $\sigma$ determines a stratum $\Ecal \subseteq \ol\Mcal_{1,\alpha}(\PP^m|H,d)$, and the morphism above restricts to a morphism
\[
\Dcal \to \Ecal.
\] 

\noindent
{\bf Splitting in the ordinary geometry.} The locus $\Ecal$ has the following description. Let $\Lambda_0$ denote the collection of stable vertices of the combinatorial type which are mapped to $0\in \RR_{\geq 0}$, and use $\Lambda_{>0}$ for the stable vertices which are mapped to $\RR_{>0}$. Each vertex $v$ has associated discrete data $\Gamma_v$ (genus, degree, contact order). The vertices over $0$ determine moduli spaces of relative stable maps, while the higher level vertices determine maps to rubber.  The stratum $\Ecal$ admits a finite splitting morphism
\begin{equation}\label{IIIa fibre product} \Ecal \xrightarrow{\rho} \left( \prod_{v \in \Lambda_0} \ol\Mcal_{\Gamma_v}(\PP^m|H) \times \prod_{v \in \Lambda_{>0}} \ol\Mcal^{\leftrightarrow}_{\Gamma_v}(\mathbb{P}|H_0+H_\infty) \right) \times_{H^{2\epsilon}} H^{\epsilon} \end{equation}
where $\epsilon$ is the number of edges of the stabilised dual graph, and the fibre product is taken over the appropriate evaluation maps. Splitting the diagonal of $H^{2\epsilon}$ we see that tautological integrals over $\Ecal$ are completely determined in terms of integrals over moduli spaces with ``smaller'' combinatorial data.\medskip

\noindent
{\bf Exportation to the reduced geometry.}  There is a virtual divisor $\widetilde\Dcal \subseteq \widetilde\VZ_{1,\alpha}(\PP^m|H,d)$ such that 
\[
\Dcal = \widetilde\Dcal \cap \VZ_{1,\alpha}(\PP^m|H,d).
\]
We wish to use the splitting of $\Ecal$ to recursively compute integrals on $\Dcal$. Noting that $\Dcal$ is cut out of the excess-dimensional space $\widetilde \Dcal$ by the factorisation properties, we make the reduction in two steps:
\begin{enumerate}
\item in \S \ref{subsection Dtilde from E} we explain how to construct $\widetilde\Dcal$ from $\Ecal$;
\item in \S \ref{subsection D from Dtilde} we explain how to express the class of $\Dcal \subseteq \widetilde\Dcal$ in terms of tautological classes.
\end{enumerate}
To compute integrals over $\Dcal$ we simply push down the appropriate classes from $\widetilde\Dcal$ to $\Ecal$ and then integrate over $\Ecal$ using the ordinary splitting axiom.

\begin{remark}[Splitting degree and vanishing order, revisited] Similarly to \S \ref{subsubsection splitting degree} above, we can calculate the degree of $\rho$ and the vanishing order $\lambda_{\Dcal}$ of the section $s_k$ along $\Dcal$. The degree of $\rho$ is given exactly as in Lemma \ref{lem:saturation}, where we obtain one such factor for each bounded edge of the tropical target. The vanishing order $\lambda_{\Dcal}$ is more delicate, but may be computed as the index of an explicit morphism of one-dimensional lattices, explained in the proof of Lemma \ref{lemma vanishing order}.
\end{remark}

\subsubsection{Compactified bundles over strata: $\widetilde\Dcal$ from $\Ecal$}\label{subsection Dtilde from E} Recall that the map $\widetilde\Dcal \to \Ecal$ is obtained by restricting the logarithmic blowup
\begin{equation*} \widetilde\VZ_{1,\alpha}(\PP^m|H,d) \to \ol\Mcal_{1,\alpha}(\PP^m|H,d)\end{equation*}
to the locus $\widetilde\Dcal$. Toroidal geometry constructs $\widetilde \Dcal$ as a compactification of a torus bundle over a locally closed stratum in $\overline\Mcal_{1,\alpha}(\PP^m|H,d)$. \medskip

%In toric geometry, the means of studying this restriction is as follows. The fan of $\widetilde\Dcal$ (respectively, $\Ecal$) is obtained by projecting the star of the ray $\tau$ (respectively, the star of the cone $\sigma$) onto the quotient lattice. The map $\widetilde\Dcal \to \Ecal$ is then the toric morphism induced by the obvious map of fans. In order to describe this morphism more explicitly, we first subdivide the fan of $\Ecal$ by taking the images of cones in the fan of $\widetilde\Dcal$. This induces a toric blowup of $\Ecal$ which we denote by $\widetilde\Ecal$. The morphism $\widetilde\Dcal \to \Ecal$ naturally factors through a morphism $\widetilde\Dcal \to \widetilde\Ecal$ which is now flat (in contrast to $\widetilde\Dcal \to \Ecal$), and  $\widetilde\Dcal$ may now be described as a compactification of a torus bundle over $\widetilde\Ecal$, with strata determined by the combinatorics of the fan.

\noindent
{\bf Logarithmic structures on strata.} Let $W\subseteq X$ be an irreducible stratum in a toroidal embedding $X$. The stratum $W$ is itself toroidal, or, equivalently, when equipped with the divisorial logarithmic structure coming from intersecting with the toroidal boundary of $X$, the resulting structure makes $W$ logarithmically smooth. Note that this differs from the logarithmic structure pulled back from $X$ along the inclusion of $W$ into $X$. 

If $X$ and $W$ are logarithmic but not toroidal, there is an analogue. There exists a strict map
\[
X\to \mathcal A_X,
\]
where $ \mathcal A_X$ is the Artin fan of $X$, which is logarithmically \'etale over a point. The logarithmic structure of $\mathcal A_X$ is divisorial in the smooth topology. The Artin fan $\mathcal A_X$ possesses closed strata. Fix an irreducible stratum $\mathcal A_W$. In identical fashion to the paragraph above, this stack $\mathcal A_W$ possesses a logarithmic structure by restricting the divisors on $\mathcal A_X$ to $\mathcal A_W$ and taking the induced logarithmic structure. This equips a logarithmic stratum $W\hookrightarrow X$ with a logarithmic structure by pullback from the Artin fan, mirroring the construction for toroidal embeddings above. 

\begin{remark}
The inclusion of a stratum $W\hookrightarrow X$ is \textit{not} logarithmic. The pullback logarithmic structure on a closed stratum includes residual information about the normal bundle of the stratum; the construction above removes these directions. By the equivalence of Artin fans with cone stacks~\cite{CavalieriChanUlirschWise}, the above construction can be combinatorialised. Given a cone $\sigma$ in a cone stack $\Sigma$, its \textbf{star} consists of all cones of $\Sigma$ that contain $\sigma$; these naturally form a new cone complex whose Artin fan gives the logarithmic structure on $W$ above.
\end{remark}

Equipped with the strata logarithmic structures, there is a logarithmic morphism
\[
\widetilde \Dcal\to \Ecal
\]
whose generic fibre is a toric variety.  \medskip

\noindent
{\bf Flattening the map.} The morphism $\widetilde \Dcal\to \Ecal$ is not flat. In order to push integrals forward, it will be convenient to flatten the map by blowing up $\Ecal$. This can be done universally~\cite{AK,Mol16}. 

\begin{figure}
 \begin{tikzpicture}
 \tikzset{cross/.style={cross out, draw=black, fill=none, minimum size=2*(#1-\pgflinewidth), inner sep=0pt, outer sep=0pt}, cross/.default={3pt}}
 
 %combinatorial type
  \draw[fill=black] (1,1) circle (2 pt);
  \draw[fill=black] (0,-1) circle (2 pt);
  \draw[->] (1,1)--node[below,blue]{$m_1$} (2,0) (0,-1)--node[below,blue]{$m_2$} (2,0) (2,0)--node[below,blue]{$m_1+m_2$} (4,0);
  \draw (1,-.5) node[cross]{};
  \draw[fill=white] (2,0) circle (2 pt);
  
  \draw[->] (2,-1) -- (2,-1.8);
  
  \draw[fill=blue] (0,-2) circle (2 pt);
  \draw[blue,->] (0,-2)--node[below=.5cm]{$\Delta_\tau$} (4,-2);
   \draw[blue] (1,-2) node[blue]{$\times$};
   \draw[blue] (2,-2) node[blue]{$\times$};
  
  \draw [red] plot [smooth cycle] coordinates {(1,1) (0,-1) (3.5,0)};
  
 %subdivision 
  \draw[->] (6,-2) --node[below=.5cm]{$\sigma$} (10,-2) node[right]{$e_1$};
  \draw[->] (6,-2) -- (6,2) node[above]{$e_2$};
  \draw[gray] (6,-2) -- (10,1) node[right]{$(m_2,m_1)=\operatorname{Im}(\tau)$};
 \end{tikzpicture}
\caption{From $\Ecal$ to $\widetilde{\Ecal}$. Typical subdivision of a two-dimensional cone $\sigma$ along the image of a ray $\tau$.}
\label{fig:Etilde}
\end{figure}
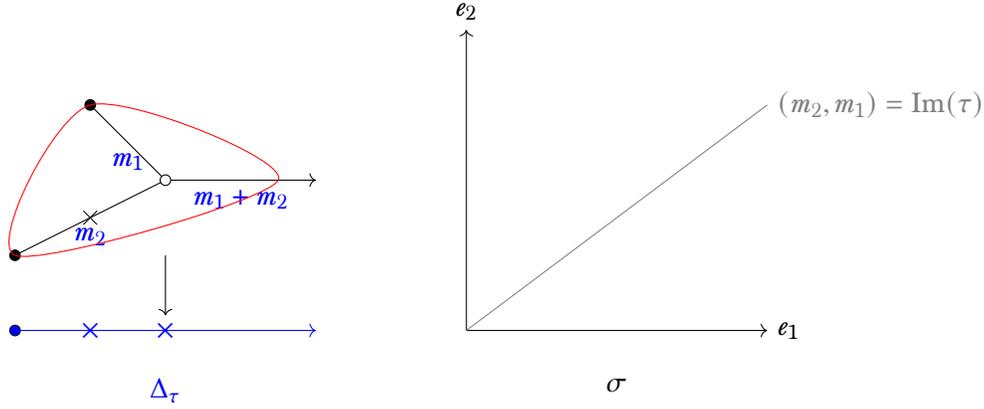

The combinatorics of the flattening can made explicit. Consider the map on tropicalizations
\[
\tau=\widetilde{\Dcal}^{\operatorname{trop}} \to \Ecal^{\operatorname{trop}}=\sigma.
\]
The tropical curves parametrised on the left inherit a central alignment from $\widetilde{\VZ}_{1,\alpha}(\PP^m|H,d)$, which orders a certain subset of the edge lengths. The polyhedral criterion for equidimensionality is that every cone of the source maps surjectively onto a cone of the target. There is a universal logarithmic modification $\widetilde\Ecal \to \Ecal$ such that $\widetilde\Dcal \to \widetilde\Ecal$ is flat, constructed by taking the minimal subdivision of $\Ecal^{\operatorname{trop}}$ such that cones map surjectively to cones (see Figure \ref{fig:Etilde}).

This subdivision may be described explicitly by employing the formalism of extended tropicalizations, see for instance~\cite[Section~2]{ACP}. The subdividing cones are obtained from cones in the star of $\tau$ under the limit $e_1=\cdots=e_r=\infty$. Cones in the star of $\tau$ are given by (in)equalities relating distances of pairs of vertices, and there are three cases to consider:
\begin{enumerate}
\item both vertices are components of the subgraph $\sqC \setminus \sqC_0$;
\item both vertices are components of the subgraph $\sqC_0$;
\item one vertex is a component of $\sqC_0$ and one is a component of $\sqC \setminus \sqC_0$.
\end{enumerate}
As we pass to the limit, the first relations impose a central alignment condition on the genus zero moduli, the second relations impose a radial alignment condition on the genus one moduli, and the third relations disappear entirely. We are therefore led to the following moduli theoretic description of the logarithmic modification $ \widetilde \Ecal$.

Keeping in mind the fibre product description of $\Ecal$, construct a modification $\widetilde \Ecal$ from $\Ecal$ as follows:
\begin{enumerate}
\item Given a logarithmic map of type $\Delta$ with splitting nodes $q_1,\ldots,q_r$, remove the genus one piece $\sqC_0$ and formally glue together the $r$ vertices of the tropicalization which carry the splitting nodes, and declare the resulting vertex to be the root of the resulting tree. Impose the central alignment condition, where the special root takes on the r\^ole previously occupied by the core.
\item Replace the fibre product factor in $\Ecal$ associated to the vertex $\sqC_0$
\begin{equation*} \ol\Mcal^{\leftrightarrow}_{\Gamma_{\tiny{\sqC_0}}}(\mathbb P|H_0+H_\infty) \end{equation*}
with the radially aligned rubber space:
\begin{equation*} \VZ^{\leftrightarrow}_{\Gamma_{\tiny{\sqC_0}}}(\mathbb P|H_0+H_\infty).\end{equation*}
(Here we are imposing fibre-direction factorisation in order to isolate the main component.)
\end{enumerate}
This procedure induces a logarithmic modification $\widetilde\Ecal \to \Ecal$ which admits a modular description as a fibre product of genus zero centrally aligned maps coming from the vertices $\sqC_1,\ldots,\sqC_r$, and genus one aligned rubber maps satisfying fibre-direction factorisation from the vertex $\sqC_0$.

%By construction, we have a map $\widetilde\Dcal \to \widetilde\Ecal$ which is a compactification of a torus bundle over $\widetilde\Ecal$. The strata of $\widetilde\Dcal$ can be read off from its tropicalization, and in particular we have the information of how (polynomials in) the classes of boundary strata push forward to $\widetilde\Ecal$.

We have produced a factorisation
\begin{equation*}\widetilde\Dcal\to\widetilde \Ecal\to\Ecal,
\end{equation*}
where the first map is flat and the second map is a logarithmic modification. $\widetilde\Dcal$ is a compactification of a torus bundle over $\widetilde\Ecal$ whose strata can be read off from the tropicalization; in particular, we have the information of how (polynomials in) the classes of boundary strata push forward to $\widetilde\Ecal$. In fact, we can say more and describe the generic fibre of $\widetilde\Dcal \to \widetilde\Ecal$:
\begin{lemma}\label{lem:generic_proj_bundle}
Generically over $\widetilde\Ecal$, the map $\widetilde \Dcal\to\widetilde\Ecal$ is a projectivised tautological vector bundle.
\end{lemma}

\begin{proof}
Let us assume that the curve does not degenerate, so that the dual graph is as prescribed by $\Delta$. In particular there are vertices $\plC_1,\ldots,\plC_r$ generically lying on the circle of radius $\delta$. The moduli of attaching data for the elliptic singularity is the complement of the coordinate hyperplanes in:
\begin{equation}\label{eqn:projectivebundle}
 \PP\left(\bigoplus_{i=1}^rT_{q_i}C_i\right).
\end{equation}
This can also be thought of as parametrising compatible identifications $\theta_{ij}\colon T_{C_i,q_i}\cong T_{C_j,q_j}$ (see \S \ref{S:ellsing}). A point in this space determines a map from the rational $r$-fold singularity to the elliptic $r$-fold singularity (this coincides with the fibre of the space of aligned curves over the space of nodal curves; see \cite[\S 3.4]{RSPW}). Along the boundary, some of the $\plC_i$ (though not all of them) may move off from the circle, so that the latter subdivides the edge between them and the core. Correspondingly, there is a destabilisation of the curve (semistable components are represented by crosses in our figures), and we obtain an elliptic singularity containing strictly semistable branches (\emph{sprouting}). This was already identified by Smyth as a compactification of the moduli space of attaching data carrying a Gorenstein universal curve \cite[\S 2.2]{SMY2}.

Tropically, pushing some of the $\plC_i$ off the circle creates new finite-length edges with coordinates $\ell_i\in\RR_{\geq_0}$ (see the right-hand side of Figure \ref{fig:off_we_go}). Since not all of them can be strictly positive, the resulting fan is that of a projective space. This describes $\widetilde\Dcal\to\widetilde\Ecal$ fibrewise over the interior.

A slight modification of the argument above is required to account for the fact that all vertices that lie on the circle and are mapped to $0\in\mathbb R_{\geq0}$ by the tropical map are already generically aligned on $\widetilde\Ecal$. Indeed, by the combinatorial description of rays in \S \ref{S:combinatorialdescription}, there are no stable vertices inside the circle, so all the edges going from a vertex of $\Lambda_0$ to the core have the same expansion factor $m$ (maximal among the contact orders at the core). Therefore, for all $i$ such that $\plC_i\in\Lambda_0$, the $T_{q_i}C_i^{\otimes m}$ can be identified to the target tangent line on the fibre product \eqref{IIIa fibre product}, and, after saturation, the $T_{q_i}C_i$ themselves are canonically isomorphic on $\widetilde\Ecal$. Thus, only one summand should be taken to represent all of them in the projective bundle \eqref{eqn:projectivebundle}.
\end{proof}

\begin{figure}
\includegraphics[scale=.7]{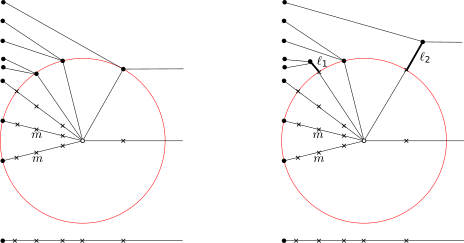}
\caption{Moving vertices off the circle creates new tropical parameters.}\label{fig:off_we_go}
\end{figure}

\subsubsection{$\Dcal$ from $\widetilde\Dcal$: first way}\label{subsection D from Dtilde} On $\widetilde\Dcal$ there is a tautological line bundle arising from the modular description of $\widetilde\VZ_{1,\alpha}(\PP^m|H,d)$, namely $\OO(\delta)$. Recall that the piecewise linear function $\delta$ may be viewed as a global section of the characteristic sheaf of the logarithmic structure on $\widetilde\VZ_{1,\alpha}(\PP^m|H,d)$. The toric geometry dictionary gives rise to an associated Cartier divisor. Alternatively, for any logarithmic stack $X$, the natural coboundary map induced by $$0\to \mathcal O_X^\star\to M_X^\text{gp}\to \overline M_X^\text{gp}\to 0$$ associates a $\mathcal O_X^\star$-torsor to a global section of the characteristic sheaf.

\begin{lemma}\label{lemma fibres of Odelta}
The fibre of $\OO(\delta)$ over a moduli point is naturally isomorphic to
\begin{equation*} T_{p}C^\prime \otimes \mathbb T \end{equation*}
where $C^\prime$ is any component of $C$ which lies on the circle of radius $\delta$, $p$ is the node separating it from the core, and $\mathbb T$ is the universal tangent line bundle on the moduli space for $\sqC_0$.
\end{lemma}

We explain the meaning of $\mathbb{T}$. Isolating the subcurve $C_0 \subseteq C$ gives rise to a morphism from $\widetilde\Dcal$ to the moduli space of radially aligned curves of genus one:
\begin{equation*} \widetilde\Dcal \to \VZ^{\text{rad}}_{1,n}.\end{equation*}
The latter space coincides with the Vakil--Zinger blow-up of $\overline \Mcal_{1,n}$ \cite[\S 2.3]{VZ}. In particular, after twisting by exceptional divisors, the tangent line bundles at two marking $\widetilde{\mathbb T}_i$ and $\widetilde{\mathbb T}_j$ are identified (via the dual Hodge bundle $\mathbb E^\vee$), extending the identification on the smooth locus provided by regarding the elliptic curve as an algebraic group. This gives rise to a single line bundle -- the ``universal tangent line bundle'' -- on $\VZ^{\text{rad}}_{1,n}$; we let $\mathbb{T}$ denote its pullback to $\widetilde\Dcal$.
\begin{proof}
Choose a component $C^\prime$ of the curve which lies on the circle. Then $\OO(\delta)$ is equal to the tensor product of tangent line bundles along the path of edges connecting $C^\prime$ to the core. However, this tensor product of line bundles is telescoping: any vertex which appears in the middle of the path will have an incoming and outgoing node (and possibly other nodes that are irrelevant for this product), and since the vertex corresponds to a smooth rational curve these tangent spaces are naturally dual to each other~\cite[\S 2.2]{VZ}. Thus these factors cancel, and the result follows.
%Note that the final term on the right hand side is not given by $\TT_{q_i} C_0$ (where $q_i$ is the appropriate splitting node) but rather the tangent space to the core. 
\end{proof}
The content of Lemma \ref{lem:generic_proj_bundle} is that the choice of a point in the fibre of $\widetilde\Dcal \to \Ecal$ provides an identification between the inward-pointing tangent lines on the circle. By Lemma \ref{lemma fibres of Odelta}, they can further be identified with the line bundle $\OO(\delta)\otimes \mathbb T^\vee$ on $\widetilde\Dcal$. %This identification amounts to a specific linear dependence relation between the tangent vectors, which gives a point in the moduli space of attaching data for the corresponding elliptic singularity \cite[\S 2.2]{SMY2}.
Corollary \ref{cor:maps_from_elliptic_sing} therefore expresses the factorisation condition as the degeneracy of the following morphism of vector bundles on $\widetilde\Dcal$:
\begin{equation*} \Sigma \operatorname{d}\!f \colon \OO(\delta)\otimes \mathbb T^\vee \to \ev_q^\st T H. \end{equation*}
Here $\Sigma \operatorname{d}\!f$ is defined by summing the images of the tangent vectors at the inward-pointing nodes on the circle, using the natural identification of each of these tangent spaces with $\OO(\delta)\otimes \mathbb T^\vee$ established above. This morphism takes values in $\ev_q^\star TH$ rather than $\ev_q^\star T\PP^m$ because the contact order of the exterior components with $H$ is strictly larger than $1$ (except for the case of lines, dealt with separately in \S \ref{S:lines}). This gives a section of the vector bundle
\begin{equation*} V = (\ev_q^\st TH) \otimes \mathbb T \otimes \OO(-\delta) \end{equation*}
which is transverse and whose vanishing locus coincides with $\Dcal$. Transversality follows by construction, because any stratum of the boundary on which $\Sigma \operatorname{d}\!f$ vanishes has been (blown up and) included into $\mathbb T \otimes \OO(-\delta)$ as a twist; in other words, the circle of radius $\delta$ always passes through at least one component which is not contracted by the map to the expansion. If the vanishing locus of $\Sigma \operatorname{d}\!f$ does not contain any stratum of the boundary, it is the closure of its intersection with the open stratum. Thus we obtain:
\begin{prop} \label{class of D} $[\Dcal] = \operatorname{e}(V) \cap [\widetilde\Dcal]$.\end{prop}
We claim that the class $\operatorname{e}(V)$ is tautological and computable. The only part which needs justification is $\OO(\delta)$. We claim that this can be obtained by pulling back a line bundle from $\Ecal$ and twisting by exceptional divisors.
\begin{lemma} For $i\in \{1,\ldots,r\}$ consider the line bundle
\begin{equation*} T_i = T_{q_i} C_i \otimes T_{q_i} C_0 \end{equation*}
on $\Ecal$, where $q_i$ is the splitting node. Then $\OO(\delta)$ is isomorphic to the pull-back of $T_i$ to $\widetilde\Dcal$ twisted by all exceptional divisors.\end{lemma}
We illustrate this lemma in the following example, which brings out the interplay between logarithmic moduli and blowups.
\begin{example}
Consider the locus $\Ecal$ given by $\sigma$ below, and the codimension-one locus $\widetilde\Dcal_\rho \subseteq \widetilde\Dcal_\sigma$ given by the two-dimensional (due to the alignment) cone $\rho$:
\begin{figure}[!htbp]
\begin{minipage}{0.3\textwidth}
\centering
\begin{tikzpicture}
%\draw [red] plot [smooth cycle] coordinates {(0,-2) (-0.5,-1) (0,0) (1,0.5) (3,-1) (1,-2.5)};

\draw[fill=black] (0,0) circle[radius=2pt];
\draw (0,0) node[above]{\small$\sqC_1$};
\draw[fill=black] (0,0) -- (2,-1);
\draw (1,-0.95) node[above]{\small$e_1$};
\draw[color=blue] (1,-0.5) node[above]{\small$2$};

\draw[fill=black] (0,-2) circle[radius=2pt];
\draw (0,-2) node[above]{\small$\sqC_2$};

\draw[->] (0,-2) -- (0,-2.5);
\draw[color=blue] (0,-2.5) node[below]{\small$0$};

\draw[fill=black] (0,-2) -- (2,-1);
\draw (1,-2) node[above]{\small$e_2$};
\draw[color=blue] (1,-1.5) node[above]{\small$2$};

\draw[fill=black,->] (2,-1) -- (3.5,-1);
\draw[color=blue] (2.35,-1) node[above]{\small$4$};

\draw[fill=white] (2,-1) circle[radius=2pt];
\draw (2,-1) node[below]{\small$\sqC_0$};

\draw[color=blue,->] (2,-2) -- (2,-3);

\draw[color=blue,->] (0,-3.5) -- (3,-3.5);
\draw[fill=blue] (0,-3.5) circle[radius=2pt];
\draw[fill=blue] (2,-3.5) circle[radius=2pt];
\end{tikzpicture}
\caption{$\sigma$}
\end{minipage}\hfill
\begin{minipage}{0.7\textwidth}
\centering
\begin{tikzpicture}
\draw [red] plot [smooth cycle] coordinates {(-4,-2) (-2,-1) (-2,1) (0,1) (3,-1) (-1,-3.5)};

\draw[fill=black] (0,0) circle[radius=2pt];
\draw (0,0) node[above]{\small$\sqC_1^0$};
\draw[fill=black] (0,0) -- (2,-1);
\draw (1,-0.95) node[above]{\small$e_1$};
\draw[color=blue] (1,-0.5) node[above]{\small$2$};

\draw[fill=black] (0,-2) circle[radius=2pt];
\draw (0,-2) node[above]{\small$\sqC_2^0$};
\draw[fill=black] (0,-2) -- (2,-1);
\draw (1,-2) node[above]{\small$e_2$};
\draw[color=blue] (1,-1.5) node[above]{\small$2$};

\draw[fill=black,->] (2,-1) -- (3.5,-1);
\draw[color=blue] (2.35,-1) node[above]{\small$4$};

\draw[fill=white] (2,-1) circle[radius=2pt];
\draw (2,-1) node[below]{\small$\sqC_0$};

\draw (-2,1) -- (0,0);
\draw[fill=black] (-2,1) circle[radius=2pt];
\draw (-2,1) node[above]{\small$\sqC_1^1$};
\draw (-1,0.5) node[below]{\small$f_1$};
\draw[color=blue] (-1,0.5) node[above]{\small$1$};

\draw (-2,1) -- (-4,1);
\draw[fill=black] (-4,1) circle[radius=2pt];
\draw (-4,1) node[above]{\small$\sqC_1^3$};
\draw (-3,1) node[below]{\small$g_1$};
\draw[color=blue] (-3,1) node[above]{\small$1$};

\draw (-2,-1) -- (0,0);
\draw[fill=black] (-2,-1) circle[radius=2pt];
\draw (-2,-1) node[above]{\small$\sqC_1^2$};
\draw (-1,-0.5) node[below]{\small$f_2$};
\draw[color=blue] (-1,-0.5) node[above]{\small$1$};

\draw (-2,-1) -- (-4,-1);
\draw[fill=black] (-4,-1) circle[radius=2pt];
\draw (-4,-1) node[above]{\small$\sqC_1^4$};
\draw (-3,-1) node[below]{\small$g_2$};
\draw[color=blue] (-3,-1) node[above]{\small$1$};

\draw (0,-2) -- (-2,-2);
\draw[fill=black] (-2,-2) circle[radius=2pt];
\draw (-2,-2) node[above]{\small$\sqC_2^1$};
\draw (-1,-2) node[below]{\small$f_3$};
\draw[color=blue] (-1,-2) node[above]{\small$2$};
\draw[->] (-2,-2) -- (-2,-2.5);
\draw[color=blue] (-2,-2.5) node[below]{\small$0$};

\draw (-2,-2) -- (-4,-2);
\draw[fill=black] (-4,-2) circle[radius=2pt];
\draw (-4,-2) node[above]{\small$\sqC_2^2$};
\draw (-3,-2) node[below]{\small$g_3$};
\draw[color=blue] (-3,-2) node[above]{\small$2$};

\draw[color=blue,->] (0,-2.5) -- (0,-3.5);

\draw[color=blue,->] (-4,-4) -- (3,-4);
\draw[fill=blue] (0,-4) circle[radius=2pt];
\draw[fill=blue] (2,-4) circle[radius=2pt];
\draw[fill=blue] (-2,-4) circle[radius=2pt];
\draw[fill=blue] (-4,-4) circle[radius=2pt];
\end{tikzpicture}
\caption{$\rho$}
\end{minipage}
\end{figure}

On the cone $\rho$ we have the following tropical continuity equations
\begin{align*}
e:=e_1 = e_2, \qquad f:=f_1 = f_2 = 2f_3, \qquad g:=g_1 = g_2 = 2g_3
\end{align*}
together with the equation imposed by the choice of alignment:
\begin{align*} e_1 + f_1 = e_2 + f_3 + g_3 \qquad (\Longleftrightarrow f = g).
\end{align*}

In this example we have:
\[\lambda(\plC^0_1)=\lambda(\plC^0_2)=e,\quad \lambda(\plC^1_1)=\lambda(\plC^2_1)=\lambda(\plC^2_2)=e+f=\delta,\quad \lambda(\plC^3_1)=\lambda(\plC^4_1)=e+2f.\]
From now on we use the same symbol to denote an edge length and the corresponding node on $\widetilde\Dcal_\rho$. On this locus we have line bundles pulled back from $\Ecal$
\begin{align*} T_1 & = T_{e_1} C_1^0 \otimes T_{e_1} C_0 \\
T_2 & = T_{e_2} C_2^0 \otimes T_{e_2} C_0\end{align*}
as well as the logarithmic line bundle
\begin{equation*} \OO(\delta) = T_{f_1} C_1^1 \otimes T_{e_1} C_0 = T_{f_2} C_1^2 \otimes T_{e_1} C_0 = T_{g_3} C_2^2 \otimes T_{e_2}C_0 
\end{equation*}
where the identifications follow from Lemma \ref{lemma fibres of Odelta}. We claim that: 
\begin{equation} \label{Odelta eqn} \OO(\delta) = T_1\otimes \OO(\widetilde\Dcal_\rho) = T_2 \otimes \OO(\widetilde\Dcal_\rho).\end{equation}
Notice that $\OO(\widetilde\Dcal_\rho) = \OO(f)=\OO(g)$. Thus we have
\begin{align*} T_1 \otimes \OO(\widetilde\Dcal_\rho) & = \OO(f_1) \otimes T_{e_1} C_1^0 \otimes T_{e_1}C_0 \\
& = T_{f_1}C_1^1 \otimes T_{f_1} C_1^0 \otimes T_{e_1} C_1^0 \otimes T_{e_1}C_0  \\
& = T_{f_1}C_1^1\otimes T_{e_1}C_0 \qquad (= \OO(\delta))
\end{align*}
where in the last line we have used the same telescoping trick used in the proof of Lemma \ref{lemma fibres of Odelta}, namely that $T_{f_1}C_1^0 \otimes T_{e_1} C_1^0 = \OO$. On the other hand, for $T_2$ we may write $f=f/2 + g/2=f_3+g_3$ in order to obtain:
\begin{equation*} T_2 \otimes \OO(\widetilde\Dcal_\rho) = \OO(g_3) \otimes \OO(f_3) \otimes T_{e_2} C_2^0 \otimes T_{e_2} C_0 = T_{g_3}C_2^2 \otimes T_{e_2}C_0 = \OO(\delta) \end{equation*}
where we have used the same telescoping trick. Thus we have proven \eqref{Odelta eqn}. This observation generalises, giving a precise relation between the intrinsically-defined line bundle $\OO(\delta)$ on $\widetilde\Dcal$ (which we use to impose the factorisation condition) and the line bundles $T_i$ pulled back from $\Ecal$. This allows us to express the Chern classes of $\OO(\delta)$ (and hence $V$) in terms of pull-backs of tautological classes from $\Ecal$ and boundary strata on $\widetilde\Dcal$, which is sufficient in order to compute integrals.
\end{example}

\subsubsection{$\Dcal$ from $\widetilde\Dcal$: second way} \label{section D from Dtilde second} %To establish a more direct connection with the geometry of the fibre product \eqref{IIIa fibre product}, recall that $\widetilde\Ecal$ is built as an explicit modification of a finite cover of \eqref{IIIa fibre product}. For $i \in\{1,\ldots,r\}$ (with the convention established at the end of Lemma \ref{lem:generic_proj_bundle}), it makes sense on $\widetilde\Ecal$ to consider the line bundle $\OO(\delta_i)$ associated to a shortest path from the core to the circle in the $\sqC_i$-direction. 
We analyse the degeneracy locus of $\Sigma\operatorname{d}\!f$, in order to have a more direct construction of $\Dcal$ from the fibre product \eqref{IIIa fibre product}. On $\Ecal$ consider the piecewise linear function $\delta_i$ ($i=1,\ldots,r$), defined as the minimal distance from a component of (i.e. generising to) the core $\plC_0$ to a non-contracted component of the $i$-th tail $\plC_i$, and the associated line bundle $\OO(\delta_i)$.
Construct
\begin{equation*}\mathcal P=\PP_{\Ecal}\left(\bigoplus_{i=1}^r\OO(\delta_i)\right)\xrightarrow{p}\Ecal.\end{equation*}
When all fibre coordinates are non-zero, a point of $\mathcal P$ induces compatible isomorphisms $\OO(\delta_i)\cong\OO(\delta_j)$ for all $i,j$. We think of this as saying that all the $\plC_i$ are at the same distance from the core. When some of the coordinates are $0$, the corresponding tail is further from the core. This corresponds to an alignment of $\plC_1,\ldots,\plC_r$ with respect to the radius $\delta$, that is the distance of the closest of them from the core. Therefore, by the universal property,
%Since $\sqC_1,\ldots,\sqC_r$ are always ordered on $\widetilde \Dcal$ (they may lie on the circle, be the first component off the circle, or degenerate and lie in the strict interior),
there is a morphism $\widetilde \Dcal\to\mathcal P$ lifting $\widetilde \Dcal\to \Ecal$. 

In order to express factorisation, consider the map of vector bundles on $\mathcal P$:
\begin{equation*}  \OO_{\mathcal P}(-1)\to p^*\left(\bigoplus_{i=1}^r\OO(\delta_i)\right)\xrightarrow{\sum\operatorname{d}\!f}\ev_q^*TH.\end{equation*}
The composite, regarded as a section $s$ of $W=\OO_{\mathcal P}(1)\otimes \operatorname{ev}_q^*(TH)$, vanishes precisely where $f$ factors through the elliptic singularity determined by the point in the fibre of $p$. Generically on $\widetilde\Ecal$, this is exactly what we are after (compare with Lemma \ref{lem:generic_proj_bundle}). On the boundary, though, it may happen that a simultaneous degeneration in the base and the fibre makes this factorisation trivial, in the sense that the map $\bar f$ is constant on all the branches of the elliptic singularity. These loci may be identified explicitly: let $\Lambda=(\Lambda_B,\Lambda_F)$ with $\Lambda_B\subseteq \{\sqC_1,\ldots,\sqC_r\}$ and $\Lambda_F=\{\sqC_1,\ldots,\sqC_r\}\setminus \Lambda_B$, and let $\Xi_\Lambda$ be the locus in $\mathcal P$ where:
\begin{itemize}
 \item[(B)] $\forall i\in\Lambda_B$, the $i$-th component of the fibre product \eqref{IIIa fibre product}, namely the map $f_i\colon C_i\to(\PP^m|H)$, degenerates so that it is constant near the attaching point to the core.
 \item[(F)] $\forall j\in\Lambda_F$, the $j$-th coordinate in the fibre of $p\colon\mathcal P\to\Ecal$ vanishes.
\end{itemize}
It should be noted that the degeneracy loci of (B) are among Gathmann's comb loci \cite{Ga}, similar to the loci $\pazocal D$ that we are considering, but in genus $0$, and therefore recursively computable.
\begin{lemma} 
 The $\Xi_\Lambda$ are precisely the boundary loci on which $s$ vanishes.
\end{lemma}

See \cite[\S 3.2]{VZ} for an analogous phenomenon in the non-relative case. Clearly, these loci may be in excess with respect to the expected codimension of $V(s)$, that is $N-1$. These loci are logarithmic strata, so there exists a logarithmic blowup that principalises them. We claim that $\widetilde\Dcal$ is one. 
\begin{lemma}
 The pullback of $\Xi_\Lambda$ is principal in $\widetilde\Dcal$.
\end{lemma}
\begin{proof}We analyse the ideal of $\Xi_\Lambda$: it is generated by
 \begin{enumerate}
  \item equations on $\Ecal$ cutting the locus where $C_i$ degenerates so that $f_i$ is constant on a neighbourhood of $q_i$ for $i\in \Lambda_B$: these correspond to the tropical functions $\delta_i-\lambda(q_i)$ where $\lambda$ is the distance function to the core;
  \item equations on the fibre of $p$ for the linear subspace corresponding to $i\in \Lambda_F$: tropically, these correspond to the $\ell_i$ discussed in Lemma \ref{lem:generic_proj_bundle}.
 \end{enumerate}
By construction, on $\widetilde \Dcal$ it is always possible to tell which of these tropical lengths is the shortest. Figure \ref{fig:principalisation} represents the situation: the circle we have drawn on $\Xi_\Lambda$ is degenerate - it passes through no non-contracted vertex - therefore we need to decide for a strictly larger circle in order to lift to $\widetilde D$. This shows that the ideal is principalised on $\widetilde \Dcal$.
\end{proof}

\begin{figure}[hbt]
\includegraphics{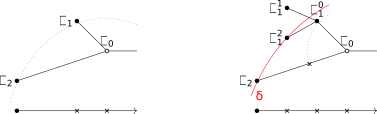}
\caption{On the left, a generic point of $\mathcal P$. On the right, a point of $\Xi$. The dotted circle is the one we draw on $\mathcal P$. The solid circle represents a radius $\delta$ enlarging the dotted one (a stratum of $\widetilde\Dcal$).}\label{fig:principalisation}
\end{figure}

Finally, on $\widetilde\Dcal$ %(or, in fact, on an intermediate blowup $\widetilde{\mathcal P}\to\mathcal P$)
there are exceptional divisors $\widetilde \Xi_\Lambda$ corresponding to $\Xi_\Lambda$ in $\mathcal P$. The $\Xi_\Lambda$ exhaust the excess-dimensional components of the vanishing of $s\in H^0(\mathcal P,W)$, so the latter induces a section $\tilde s\in H^0(\widetilde\Dcal,W(-\Sigma_\Lambda \widetilde\Xi_\Lambda))$ by pullback and twisting; the vanishing locus is dimensionally transverse to the boundary, and on a dense open coincides with the generic point of $\Dcal$, therefore $V(\tilde s)=\Dcal$.

\begin{prop}
 \[ [\Dcal]=e(W(-\Sigma_\Lambda \widetilde\Xi_\Lambda))\cap[\widetilde{\mathcal P}]\]
 where $\widetilde{\mathcal P}$ denotes the blow-up of $\mathcal P$ in the $\Xi_\Lambda$.
\end{prop}

Applying \cite[Corollary 4.2.2]{FUL}, we get:
\begin{cor} In $A_*(\mathcal P)$ the following formula holds:
 \[ [\Dcal]=\sum_{\Lambda}\sum_i s_i(\Xi_\Lambda,\mathcal P)c_{N-1-i}(\ev_q^*TH\otimes\OO_\mathcal P(1))\]
\end{cor}

It is possible to push the formula to $\Ecal$ and apply \cite[\S3.1]{FUL} to calculate $[\Dcal]$ in terms of:
\begin{itemize}
 \item evaluation classes, such as $\ev_q^*TH$;
 \item Chern classes of the logarithmic line bundles $\OO(\delta_i)$;
 \item Segre classes of Gathmann's comb loci indexed by $\Lambda_B$.
\end{itemize}

\begin{cor}
 After pushforward to $\Ecal$, the contribution of $\Dcal$ to the invariants can be computed in terms of tautological integrals on moduli spaces of maps with lower numerics.
\end{cor}

\subsection{Recursion for general $(X,Y)$}\label{section recursion for general pair} Now let $(X,Y)$ be a general smooth very ample pair. In \S \ref{subsection virtual} we studied the moduli space $\VZ_{1,\alpha}(X|Y,\beta)$ of centrally aligned maps to $(X,Y)$ and equipped it with a virtual fundamental class by diagonal pull-back from $\VZ_{1,\alpha}(\PP^m|H,d)$.

This construction of the virtual fundamental class means that once we have established a recursion formula for $(\PP^m,H)$, it pulls back to give a recursion formula for $(X,Y)$. To be more precise, Theorem \ref{theorem recursion} holds \emph{mutatis mutandis}, with the same values of $\lambda_\Dcal$. The computation of the splitting orders is exactly the same, and the recursive description of the boundary strata pulls back to give an entirely analogous description on $\VZ_{1,\alpha}(X|Y,\beta)$.

\section{Quantum Lefschetz algorithm}\label{section recursion algorithm}
\noindent Consider a smooth pair $(X,Y)$ with $Y$ very ample, and let $\mathbb{P}$ be the projective bundle $\mathbb{P}=\PP_Y(\operatorname{N}_{Y|X} \oplus\OO_Y)$. We assume that all genus zero and reduced genus one Gromov--Witten invariants of $X$ are known. From these starting data, we will apply our recursion formula to compute:
\begin{enumerate}
\item the genus one \textbf{reduced restricted absolute Gromov--Witten theory} of $Y$;
\item the genus one \textbf{reduced relative Gromov--Witten theory} of the pair $(X,Y)$;
\item the genus one \textbf{reduced rubber theory} of $\mathbb{P}$.
\end{enumerate}

\subsection{Reduced absolute, relative and rubber invariants} To be precise: by a reduced invariant of $Y$ we mean an integral over $\VZ_{1,n}(Y,\beta)$ of a product of evaluation classes and pullbacks of psi classes along morphisms forgetting subsets of the marked points:
\begin{equation*} \int_{\virt{\VZ_{1,n}(Y,\, \beta)}} \prod_{i=1}^n \fgt_{S_i}^\star \psi_i^{k_i} \cdot \ev_i^\star \gamma_i.\end{equation*}
Here $S_i \subseteq \{1,\ldots,n\}$ with $i \not\in S_i$ for all $i$; taking $S=\emptyset$ gives the ordinary  psi classes. The evaluation maps are viewed as mapping into $X$ (hence the adjective ``restricted''). Reduced relative invariants of $(X,Y)$ are defined in the same way
\begin{equation*} \int_{\virt{\VZ_{1,\alpha}(X|Y,\, \beta)}} \prod_{i=1}^n \fgt_{S_i}^\star \psi_i^{k_i} \cdot \ev_i^\star \gamma_i\end{equation*}
except now the forgetful morphism maps into a space of absolute stable maps:
\begin{equation*} \fgt_S \colon \VZ_{1,\alpha}(X|Y,\beta) \to \VZ_{1,n-\#S}(X,\beta).\end{equation*}
Note that even the case $S=\emptyset$ entails a nontrivial forgetful morphism. In particular, all the psi classes which we consider are \textbf{collapsed psi classes}, meaning that they are relative cotangent line classes for the corresponding collapsed stable map. Note that, unlike in the absolute case, in the relative case it may well be that the entire insertion is pulled back along a single forgetful map $\fgt_S$. The reduced rubber invariants of $\mathbb{P}$ are defined similarly, using collapsed psi classes.

\begin{remark}The systems of invariants defined above are equivalent to the classical systems of invariants (which do not use forgetful morphisms) by well-known topological recursion relations.\end{remark}

\begin{remark} In the following discussion we may omit the adjective ``reduced'' when discussing genus one invariants, but it is always implicit; all genus one invariants we consider will be reduced.\end{remark}

\subsection{Fictitious and true markings} Roughly speaking, we induct on the degree (meaning $Y\cdot\beta$), number of marked points and total tangency. We adopt Gathmann's point of view on the relative theory \cite{Ga}, where only a subset of the contact points with the divisor are required to be marked; this allows for ``total tangency'' strictly smaller than $Y\cdot\beta$.

In order to make sense of this in the logarithmic setting (where all contact points must be marked), we introduce the concept of \textbf{fictitious markings}. Consider a moduli space $\VZ_{1,\alpha}(X|Y,\beta)$ of reduced relative stable maps and a corresponding integrand $\gamma$. We let $F \subseteq \{1,\ldots,n\}$ be the maximal subset of marked points such that the following two conditions are satisfied:
\begin{enumerate}
\item $\alpha_i = 1$ for all $i \in F$;
\item the entire integrand $\gamma$ is pulled back along $\fgt_F$.
\end{enumerate}
This subset is uniquely determined, and its elements are referred to as \textbf{fictitious markings}. We note that $\alpha_i=1$ is not sufficient for $x_i$ to be fictitious; it must also be the case that all the insertions are pulled back along forgetting $x_i$. Markings which are not fictitious are referred to as \textbf{true markings}. When inducting on relative invariants we will always be interested in the number of true markings (as opposed to the total number of markings) and the \textbf{true tangency}
\begin{equation*} t:= \sum_{i \not\in F} \alpha_i \leq Y\cdot \beta =d\end{equation*}
as opposed to the total tangency, which is always $d$. This formalises the idea that relative invariants with non-maximal tangency $t<d$ can be obtained by adding $d-t$ fictitious markings of tangency $1$; see \cite[Lemma 1.15(i)]{Ga}.

\subsection{Structure of the recursion} Given the genus zero Gromov--Witten theory of $X$, the results of \cite{Ga} give an effective algorithm to reconstruct the genus zero theories of $Y$ and $(X,Y)$; the genus zero rubber theory of $\mathbb{P}$ is equal to the genus zero theory of $Y$ \cite{GathmannThesis}. Thus we may assume that all genus zero invariants are known. We assume in addition that we know the genus one reduced theory of $X$. The structure of the recursion is then as follows:\bigskip

\begin{algorithm}[H]
\DontPrintSemicolon
\For{$d \geq 0$}{
\For{$n \geq 0$}{
\For{$t \geq 0$\medskip}{
\textbf{\, Step 1: } Compute forgetful relative invariants of $(X,Y)$ (degree~$d$, $n+1$ true markings, true tangency $t$); see below.
}
\medskip \textbf{Step 2: } Compute absolute invariants of $Y$ (degree $d$, $n$ markings).\;
\For{$t \geq 0$\medskip}{
\textbf{\, Step 3: } Compute relative invariants of $(X,Y)$ (degree $d$, $n$ true markings, true tangency $t$).\;
}
}
\For{$n \geq 0$\medskip}{
\For{$m \geq 0$\medskip}{
\textbf{\, Step 4: } Compute rubber invariants of $\mathbb{P}$ (degree $d$, $n$ relative markings, $m$ non-relative markings).
}
}
}
\end{algorithm}\bigskip

\noindent Here each loop begins at $0$ and proceeds through successively larger values, with the invariants at each step expressed recursively in terms of invariants computed in the preceding steps and genus zero invariants. Although the loops involving $d, n$ and $m$ have infinite length, in order to compute any single invariant it is only necessarily to iterate the preceding loops for a finite amount of time.

\begin{remark}[Forgetful relative invariants] In Step 1 above, a \textbf{forgetful relative invariant} of $(X,Y)$ is by definition a reduced relative invariant for which there exists a marking $x_0$ such that all of the insertions are pulled back along $\fgt_{x_0}$. In our recursion, we first deal with the special class of forgetful relative invariants (with $n+1$ true markings), before computing the absolute invariants (with $n$ markings) and then returning to compute all of the relative invariants (with $n$ true markings). The need to separately treat a particular subclass of the relative invariants is an inescapable feature of the genus one recursion.\end{remark}

\begin{remark}[The base case]
The base terms of the recursion all have $d=0$ and are easily computed: the relative invariants of $(X,Y)$ are absolute invariants of $X$ (i.e. without specified tangency conditions), and the absolute invariants of $Y$ are given by obstruction bundle integrals over Deligne--Mumford spaces. The rubber invariants of $\mathbb{P}$ in degree zero are given by integrals over the main component of the double ramification cycle (in the rubber case there may be markings with tangency even for $d=0$, since $d$ only records the degree covering the base); this can likewise be computed as an integral over Deligne--Mumford space, by correcting the formula for the ordinary double ramification cycle \cite{Hain,JPPZ} by the obvious contribution from the boundary irreducible component (the observation that enables this in genus one is that the rubber moduli space is pure-dimensional, and the virtual class is simply the sum of the fundamental classes of the components). We will now explain how to perform each of the four inductive steps outlined above.\end{remark}

\begin{notation}Given tuples $\mathbf{a}=(a_1,\ldots,a_k)$ and $\mathbf{b}=(b_1,\ldots,b_k)$, we say that $\mathbf{a}<\mathbf{b}$ if there exists an $i \in \{1,\ldots,k\}$ such that $a_j = b_j$ for $j < i$ and $a_i < b_i$.\end{notation}

\subsection*{Step 1} Suppose we are given a forgetful relative invariant to compute. That is, we have a relative space $\VZ_{1,\alpha}(X|Y,\beta)$ of degree $d$, with $n+1$ true markings and true tangency $t$, and a marking $x_0$ such that the insertion $\gamma=\fgt_{x_0}^\star \delta$ for some class $\delta$. We assume inductively that every absolute, relative and rubber invariant with $(d^\prime,n^\prime) < (d,n)$ has been computed, as well as every forgetful relative invariant with $(d^\prime,n^\prime,t^\prime) < (d,n+1,t)$. As noted above, we may assume that $d \geq 1$ which means that there exists a true marking $x_1$ with $\alpha_1 \geq 1$. We decrease the tangency at $x_1$. Consider the vector of tangency orders
\begin{equation*} \alpha - e_1 \end{equation*}
obtained from $\alpha$ by replacing the entry $\alpha_1$ by $\alpha_1-1$. Consider then the moduli space:
\begin{equation*} \VZ_{1,(\alpha-e_1) \cup (1)}(X|Y,\beta). \end{equation*}
Denote the newly-introduced marking by $y$ and take the integrand $\tilde\gamma=\fgt_y^\st \gamma$. Applying our recursion formula to $x_1$, we obtain:
\begin{equation}\label{step 1 recursion}\left( (\alpha_1-1)\psi_1 + \ev_1^\st Y\right) \tilde\gamma \cap [\VZ_{1,(\alpha-e_1) \cup (1)}(X|Y,\beta)] = \tilde\gamma \cap [\Dcal(1)].\end{equation}
We examine the left-hand side. The class $\psi_1$ differs from $\fgt_y^\st \psi_1$ by the locus where $x_1$ and $y$ are contained on a contracted rational bubble. This locus consists of stable maps of the form \medskip

\begin{center}
\begin{tikzpicture}[scale=1]
%edge
\draw (0,0) to (2,0);
\draw [color=blue] (1,0) node[above]{\tiny$\alpha_1$};

%C_1
\draw [fill=white] (0,0) circle[radius=3pt];
\draw (0,-0.1) node[below]{\small$d$};
\draw (0,0) node[above]{\small$\sqC_1$};

%C_0
\draw [fill=black] (2,0) circle[radius=3pt];
\draw (2,-0.1) node[below]{\small$0$};
\draw (2,0) node[above]{\small$\sqC_0$};

%x_1
\draw [->] (2,0.1) -- (4,0.1);
\draw (3.9,0.1) node[above]{\small$x_1$};
\draw [color=blue](3,0.1) node[above]{\tiny$\alpha_1-1$};

%y
\draw [->] (2,-0.1) -- (4,-0.1);
\draw (3.9,-0.1) node[below]{\small$y$};
\draw [color=blue] (3,-0.1) node[below]{\tiny$1$};

%down arrow
\draw [color=blue,->] (1,-0.3) -- (1,-0.7);

%target
\draw [color=blue,->] (0,-1) to (4,-1);
\draw [fill=blue,color=blue] (0,-1) circle[radius=3pt];
%\draw [color=blue] (2,-1) node[right]{\small$\Sigma(X|Y)$};
\draw[color=blue] (2,-1) node{$\times$};
\end{tikzpicture}
\end{center}
with all other marked points contained on $\sqC_1$. This is isomorphic to $\VZ_{1,\alpha}(X|Y,\beta)$ and when we restrict $\tilde\gamma$ to this locus we obtain precisely the class $\gamma$ which we started with. Thus \eqref{step 1 recursion} may be written as
\begin{equation} \label{step 1 recursion 2} (\alpha_1-1) I + \fgt_y^\st \left( (\alpha_1-1)\psi_1 + \ev_1^\st Y\right)  \tilde\gamma \cap [\VZ_{1,(\alpha-e_1)\cup(1)}(X|Y,\beta)] = \tilde\gamma \cap [\Dcal(1)]. \end{equation}
where $I$ is the invariant we are trying to compute. We now show that the other terms may be expressed in terms of invariants computed earlier in the recursion.

The second term on the left-hand side is a forgetful relative invariant with the same degree and number of true markings (since $y$ is fictitious), and strictly smaller true tangency; hence it has been computed at an earlier step in the recursion. We now examine the right-hand side of \eqref{step 1 recursion 2}. Recall that all of the genus zero data has already been computed, so we only need to focus on the genus one pieces. First consider the type I loci. The genus one piece has strictly smaller degree (and so has already been computed) except in the following situation (with some stable distribution of the remaining markings):
\begin{center}
\begin{tikzpicture}[scale=1]
\draw (0,0) to (2,0);
\draw [fill=white] (0,0) circle[radius=3pt];
\draw (0,-0.1) node[below]{\small$d$};
\draw (0,0) node[above]{\small$\sqC_1$};
\draw [fill=black] (2,0) circle[radius=3pt];
\draw (2,-0.1) node[below]{\small$0$};
\draw (2,0) node[above]{\small$\sqC_0$};
\draw [->] (2,0) -- (4,0);
\draw (3.9,0) node[right]{$x_1$};

\draw [color=blue,->] (1,-0.3) -- (1,-0.7);

\draw [color=blue,->] (0,-1) to (4,-1);
\draw [fill=blue,color=blue] (0,-1) circle[radius=3pt];
%\draw [color=blue] (2,-1) node[right]{\small$\Sigma(X|Y)$};
\draw[color=blue] (2,-1) node{$\times$};
\end{tikzpicture}
\end{center}
In this situation, $C_1$ contains at most $n+1$ true markings. We consider three possible cases:
\begin{enumerate}
\item If $C_1$ has $n-1$ or fewer true markings, then its contribution has already been computed earlier in the recursion.
\item If $C_1$ has exactly $n$ true markings, then $C_0$ contains exactly one true marking besides $x_1$ (which itself may be true or fictitious). In this situation we must have $y \in C_1$, since otherwise we would have a moduli space for $C_0$ given by $\ol\Mcal_{0,k}$ with $k \geq 4$, and applying the projection formula with $\fgt_y$ we would conclude that the contribution is zero. Thus we have $y \in C_1$, and so the genus one contribution is a forgetful relative invariant with $n$ true markings, and so has already been computed.
\item Finally, if $C_1$ contains exactly $n+1$ true markings, then the only additional markings on $C_0$ are fictitious, and by the same argument as in the previous paragraph there can only be one. Thus for each fictitious marked point we obtain a locus isomorphic to $\VZ_{1,\alpha}(X|Y,\beta)$ and $\tilde\gamma$ restricts to $\gamma$ here (from the point of view of computing invariants, the fictitious marked points are indistinguishable, meaning that the contributions are all the same). Thus for each fictitious marked point (of which there is at least one, namely $y$) we get a contribution of $\alpha_1 I$ to the right-hand side of \eqref{step 1 recursion 2}, where $I$ is the invariant we are trying to compute.
\end{enumerate}\medskip
The contributions of the type II loci only involve genus zero data and can be computed algorithmically following \cite{Ga}. The contributions of the type $\dag$ loci are determined by genus zero data and tautological integrals on genus one Deligne--Mumford spaces, which can be computed algorithmically using the string and dilaton equations.

It remains to consider type III loci. If the degree of the genus one piece is less than $d$ then we have a rubber invariant of strictly smaller degree. The other possibility is that the entire curve is mapped into the divisor. In this case we apply the projection formula with $\fgt_y$ to identify this with an integral over $\VZ_{1,m}(Y,\beta)$ for some (possibly large) number $m$ of marked points. But by assumption there is another marked point $x_0$ with all of the insertions pulled back along $\fgt_{x_0}$, so a further application of the projection formula shows that this contribution vanishes. To conclude, we rearrange \eqref{step 1 recursion 2} to obtain an expression of the form
\begin{equation*} \lambda I = \text{polynomial in previously computed invariants} \end{equation*}
with $\lambda$ an explicit nonzero scalar; we have determined the forgetful relative invariant $I$, completing Step 1.

\subsection*{Step 2} Consider now the absolute space $\VZ_{1,n}(Y,\beta)$ with an insertion $\gamma$, and suppose inductively that we have computed all forgetful relative invariants with $(d^\prime,n^\prime) \leq (d,n+1)$, all relative invariants and absolute invariants with $(d^\prime,n^\prime) < (d,n)$, and all rubber invariants with $d^\prime < d$. \medskip

\noindent \emph{Step 2a.} Consider the following moduli space with $n+1$ markings:
\begin{equation*} \VZ_{1,(d,0,\ldots,0)}(X|Y,\beta). \end{equation*}
Let $x_0$ denote the relative marking and consider the integrand $\tilde\gamma= \fgt_{x_0}^\st \gamma$. Now recurse at $x_1$ to obtain:
\begin{equation}\label{step 2 recursion} \ev_1^\st Y \cdot \tilde\gamma \cap [\VZ_{1,(d,0,\ldots,0)}(X|Y,\beta)] = \tilde\gamma\cap[\Dcal(1)].\end{equation}
The left-hand side is a forgetful relative invariant of degree $d$ and $\leq n+1$ true markings, and so has already been computed.

On the right-hand side, similar arguments to those in Step~1 show that the contributions of type~II and type~$\dag$ loci can be expressed in terms of invariants computed earlier in the recursion. Similarly, the contributions of type III loci have already been computed, except for the locus where the entire curve is mapped into the divisor. On this locus we apply the projection formula to $\fgt_{x_0}$ to identify the contribution with
\begin{equation*} d^2 \cdot \gamma \cap [\VZ_{1,n}(Y,\beta)] = d^2 \cdot I \end{equation*}
where $I$ is the invariant we are trying to compute. Finally, consider loci of type I. The genus one contributions from each locus have $(d^\prime,n^\prime) < (d,n)$ (and hence have already been computed) except in the following case
\begin{center}
\begin{tikzpicture}[scale=1]
%edge
\draw (0,0) to (2,0);
\draw [color=blue] (1,0) node[above]{\tiny$d$};

%C_1
\draw [fill=white] (0,0) circle[radius=3pt];
\draw (0,-0.1) node[below]{\small$d$};
\draw (0,0) node[above]{\small$\sqC_1$};

%C_0
\draw [fill=black] (2,0) circle[radius=3pt];
\draw (2,-0.1) node[below]{\small$0$};
\draw (2,0) node[above]{\small$\sqC_0$};

%x_1
\draw [->] (2,0.1) -- (4,0.1);
\draw (3.9,0.1) node[above]{\small$x_1$};
\draw [color=blue](3,0.1) node[above]{\tiny$0$};

%y
\draw [->] (2,-0.1) -- (4,-0.1);
\draw (3.9,-0.1) node[below]{\small$x_0$};
\draw [color=blue] (3,-0.1) node[below]{\tiny$d$};

%down arrow
\draw [color=blue,->] (1,-0.3) -- (1,-0.7);

%target
\draw [color=blue,->] (0,-1) to (4,-1);
\draw [fill=blue,color=blue] (0,-1) circle[radius=3pt];
\draw[color=blue] (2,-1) node{$\times$};
%\draw [color=blue] (2,-1) node[right]{\small$\Sigma(X|Y)$};
\end{tikzpicture}
\end{center}
which contributes
\begin{equation*} d\cdot \gamma \cap [\VZ_{1,(d,\underbrace{0,\ldots,0}_{n-1})}(X|Y,\beta)]\end{equation*}
to the right-hand side of \eqref{step 2 recursion}. Here $\gamma$ is the insertion we started with; the difference is that we are now considering relative maps to $(X,Y)$ with maximal tangency at $x_1$, rather than absolute maps to $Y$. Putting this all together, we obtain from \eqref{step 2 recursion}
\begin{equation}\label{step 2 recursion 2} I = (-\gamma/d) \cap [\VZ_{1,(d,\underbrace{0,\ldots,0}_{n-1})}(X|Y,\beta)] + \text{polynomial in previously computed invariants} \end{equation}
where on the right-hand side there are $n-1$ non-relative markings $x_2,\ldots,x_n$, and a relative marking $x_1$.\medskip

\noindent \emph{Step 2b.} We now apply the recursion to the right-hand side of \eqref{step 2 recursion 2}. Consider the space
\begin{equation*} \VZ_{1,(d-1,1,0,\ldots,0)}(X|Y,\beta) \end{equation*}
where $x_1$ now has tangency $d-1$ and we have introduced a new marking $y$ with tangency $1$. As usual we take the insertion $\tilde\gamma=\fgt_y^\st \gamma$. Recursing at $x_1$ we obtain:
\begin{equation}\label{step 2 recursion 3} \left( (d-1)\psi_1 + \ev_1^\st Y \right)\cdot(-\tilde\gamma/d) \cap [\VZ_{1,(d-1,1,0,\ldots,0)}(X|Y,\beta)] = (-\tilde\gamma/d) \cap [\Dcal(1)].\end{equation}
The difference between $\psi_1$ and $\fgt_y^\st \psi_1$ is given by the locus where $x_1$ and $y$ belong to a collapsed rational bubble:
\begin{center}
\begin{tikzpicture}[scale=1]
%edge
\draw (0,0) to (2,0);
\draw [color=blue] (1,0) node[above]{\tiny$d$};

%C_1
\draw [fill=white] (0,0) circle[radius=3pt];
\draw (0,-0.1) node[below]{\small$d$};
\draw (0,0) node[above]{\small$\sqC_1$};

%C_0
\draw [fill=black] (2,0) circle[radius=3pt];
\draw (2,-0.1) node[below]{\small$0$};
\draw (2,0) node[above]{\small$\sqC_0$};

%x_1
\draw [->] (2,0.1) -- (4,0.1);
\draw (3.9,0.1) node[above]{\small$y$};
\draw [color=blue](3,0.1) node[above]{\tiny$1$};

%y
\draw [->] (2,-0.1) -- (4,-0.1);
\draw (3.9,-0.1) node[below]{\small$x_1$};
\draw [color=blue] (3,-0.1) node[below]{\tiny$d-1$};

%down arrow
\draw [color=blue,->] (1,-0.3) -- (1,-0.7);

%target
\draw [color=blue,->] (0,-1) to (4,-1);
\draw [fill=blue,color=blue] (0,-1) circle[radius=3pt];
\draw[color=blue] (2,-1) node{$\times$};
%\draw [color=blue] (2,-1) node[right]{\small$\Sigma(X|Y)$};
\end{tikzpicture}
\end{center}
The contribution of this locus to the left-hand side of \eqref{step 2 recursion 3} is:
\begin{equation*} (d-1)\cdot(-\gamma/d) \cap [\VZ_{1,(d,\underbrace{0,\ldots,0}_{n-1})}(X|Y,\beta)].\end{equation*}
What remains on the left-hand side  of \eqref{step 2 recursion 3} is a forgetful relative invariant with degree $d$ and $\leq n+1$ true markings (here $y$ is the ``forgetful marking''); this has been computed earlier in the recursion. On the right-hand side, the type I loci are recursively known except possibly in the following special cases (with some stable distribution of the remaining non-relative markings):
\begin{center}
\begin{minipage}{0.4\textwidth}
\begin{tikzpicture}[scale=1]
%edge
\draw (0,-0.1) to (2,-0.1);
\draw [color=blue] (1,-0.1) node[below]{\tiny$d-1$};

%C_1
\draw [fill=white] (0,0) circle[radius=3pt];
\draw (0,-0.1) node[below]{\small$d$};
\draw (0,0) node[above]{\small$\sqC_1$};

%C_0
\draw [fill=black] (2,0) circle[radius=3pt];
\draw (2,-0.1) node[below]{\small$0$};
\draw (2,0) node[above]{\small$\sqC_0$};

%x_1
\draw [->] (0,0.1) -- (1.5,0.1);
\draw (1.5,0.1) node[above]{\small$y$};
\draw [color=blue](0.75,0.1) node[above]{\tiny$1$};

%y
\draw [->] (2,0) -- (4,0);
\draw (3.9,0) node[below]{\small$x_1$};
\draw [color=blue] (3,0) node[below]{\tiny$d-1$};

%down arrow
\draw [color=blue,->] (1,-0.5) -- (1,-0.9);

%target
\draw [color=blue,->] (0,-1) to (4,-1);
\draw [fill=blue,color=blue] (0,-1) circle[radius=3pt];
\draw[color=blue] (2,-1) node{$\times$};

%\draw [color=blue] (2,-1) node[right]{\small$\Sigma(X|Y)$};
\draw (2,-1.3) node[below]{\small{Case 1}};
\end{tikzpicture}
\end{minipage}
\begin{minipage}{0.4\textwidth}
\begin{tikzpicture}[scale=1]
%edge
\draw (0,0) to (2,0);
\draw [color=blue] (1,0) node[above]{\tiny$d$};

%C_1
\draw [fill=white] (0,0) circle[radius=3pt];
\draw (0,-0.1) node[below]{\small$d$};
\draw (0,0) node[above]{\small$\sqC_1$};

%C_0
\draw [fill=black] (2,0) circle[radius=3pt];
\draw (2,-0.1) node[below]{\small$0$};
\draw (2,0) node[above]{\small$\sqC_0$};

%x_1
\draw [->] (2,0.1) -- (4,0.1);
\draw (3.9,0.1) node[above]{\small$y$};
\draw [color=blue](3,0.1) node[above]{\tiny$1$};

%y
\draw [->] (2,-0.1) -- (4,-0.1);
\draw (3.9,-0.1) node[below]{\small$x_1$};
\draw [color=blue] (3,-0.1) node[below]{\tiny$d-1$};

%down arrow
\draw [color=blue,->] (1,-0.3) -- (1,-0.7);

%target
\draw [color=blue,->] (0,-1) to (4,-1);
\draw [fill=blue,color=blue] (0,-1) circle[radius=3pt];
\draw[color=blue] (2,-1) node{$\times$};

%\draw [color=blue] (2,-1) node[right]{\small$\Sigma(X|Y)$};

\draw (2,-1.3) node[below]{\small{Case 2}};
\end{tikzpicture}
\end{minipage}
\end{center}
We deal with these cases individually:
\begin{itemize}
\item In Case 1, the contribution from $C_1$ is a forgetful relative invariant with $\leq n$ true markings, and hence has previously been computed. 
\item In Case 2, we first note that there cannot be any more markings on $C_0$ (since otherwise applying the projection formula with $\fgt_y$ shows that the contribution vanishes). We thus obtain a single locus, which contributes precisely:
\begin{equation*} d\cdot(-\gamma/d) \cap [\VZ_{1,(d,\underbrace{0,\ldots,0}_{n-1})}(X|Y,\beta)] = -\gamma\cap[\VZ_{1,(d,\underbrace{0,\ldots,0}_{n-1})}(X|Y,\beta)].\end{equation*}
\end{itemize}
As in Step~1, we see that the type II and type $\dag$ contributions have already been computed earlier in the algorithm. The only type III locus whose contribution has not previously been computed is when the entire curve is mapped into the divisor. In this case we apply the projection formula with $\fgt_y$ to calculate its contribution as:
\begin{equation*} (-\gamma/d) \cap [\VZ_{1,n}(Y,\beta)] = -I/d.\end{equation*}
Thus \eqref{step 2 recursion 3} gives:
\begin{equation*} -(\gamma/d) \cap [\VZ_{1,(d,\underbrace{0,\ldots,0}_{n-1})}(X|Y,\beta)] = I/d + \text{polynomial in previously computed invariants}.\end{equation*}
Substituting into \eqref{step 2 recursion 2}, we obtain
\begin{equation*} I\cdot(1-d^{-1}) = \text{polynomial in previously computed invariants} \end{equation*}
which completes the recursion step as long as $d \neq 1$. But since $\VZ_{1,n}(H,1)=\emptyset$ as noted in Remark~\ref{rem: constant-maps}, it follows that $\VZ_{1,n}(Y,\beta)=\emptyset$ if $d=Y\cdot\beta=1$, so we may assume $d \geq 2$ in this step. Step 2, computing the reduced absolute invariant of $Y$ of degree $d$ and $n$ markings, is complete.

\subsection*{Step 3} Now suppose we are given a relative space $\VZ_{1,\alpha}(X|Y,\beta)$ with an insertion $\gamma$, and suppose inductively that we have computed all forgetful relative invariants with $(d^\prime,n^\prime) \leq (d,n+1)$, all absolute invariants with $(d^\prime,n^\prime) \leq (d,n)$, all relative invariants with $(d^\prime,n^\prime,t^\prime)<(d,n,t)$ and all rubber invariants with $d^\prime < d$.

Choose a true marking $x_1$ with $\alpha_1 \geq 1$ and consider the moduli space
\begin{equation*} \VZ_{1,(\alpha-e_1)\cup(1)}(X|Y,\beta) \end{equation*}
where $y$ is the newly-introduced marking. As usual consider the insertion $\tilde\gamma=\fgt_y^\star \gamma$. Recursing at $x_1$ we obtain:
\begin{equation}\label{step 3} \left( (\alpha_1-1)\psi_1 + \ev_1^\st H\right) \tilde\gamma \cap [\VZ_{1,(\alpha-e_1)\cup(1)}(X|Y,\beta)] = \tilde\gamma \cap [\Dcal(1)].\end{equation}
The left-hand side is a relative invariant with the same degree and number of true markings, but smaller true tangency; hence it has already been computed in the recursion. On the right-hand side, the type I contributions have already been computed except for those of the following form
\begin{center}
\begin{tikzpicture}[scale=1]
\draw (0,0) to (2,0);
\draw [color=blue] (1,-0.05) node[above]{\tiny$\alpha_1$};
\draw [fill=white] (0,0) circle[radius=3pt];
\draw (0,-0.1) node[below]{\small$d$};
\draw (0,0) node[above]{\small$\sqC_1$};
\draw [fill=black] (2,0) circle[radius=3pt];
\draw (2,-0.1) node[below]{\small$0$};
\draw (2,0) node[above]{\small$\sqC_0$};
\draw [->] (2,0) -- (4,0);
\draw (3.9,0) node[right]{$x_1$};
\draw [color=blue] (3,-0.05) node[above]{\tiny$\alpha_1-1$};

\draw [color=blue,->] (1,-0.3) -- (1,-0.7);

\draw [color=blue,->] (0,-1) to (4,-1);
\draw [fill=blue,color=blue] (0,-1) circle[radius=3pt];
\draw[color=blue] (2,-1) node{$\times$};
%\draw [color=blue] (2,-1) node[right]{\small$\Sigma(X|Y)$};
\end{tikzpicture}
\end{center}
where $C_0$ contains a single fictitious marking (if it had multiple fictitious markings then the contribution would vanish by projection formula) and all the other markings are on $C_1$. There is one such locus for each fictitious marking, and each contributes $\alpha_1 I$ where $I$ is the invariant we are trying to compute. Note that $\alpha_1 \neq 0$, and that this term appears at least once since $y$ is a fictitious marking; so we obtain a nonzero multiple of $I$.

Using the same arguments as in Step~1, we see that the type II and $\dag$ contributions have already been computed. The type III contributions are determined by lower-degree rubber invariants, except for when the whole curve is mapped into $Y$; however in this case we may apply the projection formula with $\fgt_y$ to identify the contribution with an absolute invariant of $Y$ with degree $d$ and $n$ markings, which has also been computed earlier (in Step 2). We therefore see that \eqref{step 3} may be rearranged to give:
\begin{equation*} I = \text{polynomial in previously computed invariants}.\end{equation*}
We have thus computed the reduced relative invariant of $(X,Y)$ with degree $d$, $n$ true markings and true tangency $t$, which completes Step 3.

\subsection*{Step 4} For the final step, consider a rubber space $\VZ^{\leftrightarrow}_{1,\alpha}(\mathbb P|Y_0+Y_\infty,\beta)$ with insertion $\gamma$. Suppose inductively that we have computed all absolute and relative invariants with $d^\prime \leq d$ and all rubber invariants with $(d^\prime,n^\prime,m^\prime) < (d,n,m)$.

We note the following important reduction: if there exists a relative marking $x_k$ such that all insertions are pulled back along $\fgt_{k}$, then we may apply the projection formula, together with the following identity (arising from the torsion in the Jacobian of an elliptic curve)
\begin{equation*} (\fgt_{k})_\st [\VZ^{\leftrightarrow}_{1,\alpha}(\mathbb P|Y_0+Y_\infty,\beta)] = \alpha_k^2 \cdot [\VZ_{1,n+m-1}(Y,\beta)] \end{equation*}
to identify the rubber invariant with a multiple of a reduced invariant of $Y$, which has the same degree and hence has been computed earlier.

We will deal with the general case by reducing to the one above. Consider the moduli space
\begin{equation*}\label{step 4 recursion space} \VZ^{\leftrightarrow}_{1,\alpha\cup (0)}(\mathbb P|Y_0+Y_\infty,\beta)\end{equation*}
obtained by introducing a marked point $y$ with no tangency. Let $x_1$ be a positive-tangency marking (such a marking always exists since $d \geq 2$) and let $\tilde\gamma$ be the insertion obtained from $\gamma$ by replacing $\ev_1$ and $\psi_1$ by $\ev_y$ and $\psi_y$, and then pulling back along $\fgt_1$.

We make use of a recursion formula for rubber spaces analogous to the recursion formula for relative spaces used in Steps 1--3. Following \cite{EKatz}, there is a line bundle
\begin{equation*} L_y^{\not\in \operatorname{bot}} \text{ on } \VZ^{\leftrightarrow}_{1,\alpha\cup (0)}(\mathbb P|Y_0+Y_\infty,\beta)\end{equation*}
together with a section $s_y^{\not\in\operatorname{bot}}$ whose vanishing locus consists of the locus $\Dcal(y)$ of rubber maps where $y$ is not mapped into the bottom level of the expanded target. As in \S \ref{Section Gathmann line bundle}, we can give a logarithmic interpretation of this: it corresponds to the piecewise-linear function on the tropical moduli space which associates, to every rubber tropical map, the distance between $\varphi(\sqC_y)$ and the leftmost vertex of the tropical target, where $\sqC_y$ is the vertex of the source curve containing the flag corresponding to $y$. Using this tropical description, we calculate the vanishing orders of $s_y^{\not\in\operatorname{bot}}$ along the various components of $\Dcal(y)$, and show that
\begin{equation*}\cchern_1(L_y^{\not\in\operatorname{bot}}) = \Psi_0 - \ev_y^\st Y\end{equation*}
where $\Psi_0$ is a target psi class (see \cite{EKatzLB} for an analogue in the standard theory). From this, we obtain a rubber recursion formula
\begin{equation}\label{step 4 recursion} (\Psi_0 - \ev_y^\st Y)\cdot \tilde\gamma \cap [\VZ^{\leftrightarrow}_{1,\alpha\cup(0)}(\mathbb P|Y_0+Y_\infty,\beta)] = \tilde\gamma \cap [\Dcal(y)]\end{equation}
where the fundamental class $[\Dcal(y)]$ is weighted by vanishing orders on the components.\medskip

\noindent \emph{Step 4a.} We will first show that the left-hand side can be expressed in terms of previously computed invariants. By construction the classes $\tilde\gamma$ and $\ev_y^\star Y$ are both pulled back along $\fgt_{1}$. It remains to examine $\Psi_0$. If there exists a negative-tangency marking $x_2$, then we have \cite[Construction 5.1.17]{GathmannThesis}
\begin{equation*}\label{Psi0 formula} \Psi_0 = -\alpha_2 \hat\psi_2 - \ev_2^\st Y \end{equation*}
where $\hat\psi_2$ is a \textbf{non-collapsed} psi class. (If there are no negative-tangency markings, then the construction given in \cite[\S 1.5.2]{MaulikPandharipande} shows that $\Psi_0=0$.)

Now, $\hat\psi_2 - \psi_2$ is given by the locus where $x_2$ belongs to a trivial bubble. This entails a splitting of the curve into pieces, each of which contributes a rubber integral. Typically each such piece will have $(g^\prime,d^\prime,n^\prime,m^\prime) < (1,d,n,m)$ and hence has previously been computed. The one exception is when all of the genus and degree is concentrated on the top level of the expansion, with the bottom level containing only a single non-relative marking in addition to all the negative-tangency markings. But in this case the contribution is a rubber invariant where all of the insertions are pulled back along $\fgt_1$, and hence we may apply our fundamental reduction, using the projection formula to identify the rubber invariant with (a multiple of) an absolute invariant of $Y$, which has previously been computed. We conclude that, up to previously computed terms, we may replace $\hat\psi_2$ by $\psi_2$ in the left-hand side of \eqref{step 4 recursion}.

The difference between $\psi_2$ with $\fgt_1^\st \psi_2$ is given by the locus where $x_1$ and $x_2$ belong to a collapsed rational piece. The contribution of this locus consists of rubber invariants with strictly fewer relative markings, and hence is recursively known. Up to previously computed invariants, \eqref{step 4 recursion} becomes
\begin{equation}\label{step 4 recursion 2} \fgt_1^\st (-\alpha_2\psi_2 - \ev_y^\st Y)\tilde\gamma \cap [\VZ^{\leftrightarrow}_{1,\alpha\cup(0)}(\mathbb P|Y_0+Y_\infty,\beta)] = \tilde\gamma\cap[\Dcal(y)]\end{equation}
and the left side can be expressed via previously computed invariants, by the projection formula. \medskip

\noindent \emph{Step 4b.} Let us now examine the right-hand side of \eqref{step 4 recursion 2}. The components of $\Dcal(y)$ are indexed by splittings of the curve, and the contributions can be expressed in terms of previously computed invariants unless there is a piece of the curve carrying all of the genus and degree. We therefore restrict to considering these cases. Denote subcurve carrying all of the genus and degree by $C^\prime\subseteq C$.
\begin{enumerate}
\item If $C^\prime$ is mapped to top level, then either it contains $x_1$, in which case we apply $\fgt_1$ to compute the contribution, or it does not contains $x_1$, in which case it has fewer than $n$ relative markings and is known recursively. 
\item  If on the other hand $C^\prime$ is not mapped to top level, then generically it is mapped to the bottom level (since generically the core is not contracted on this locus, so the desingularisation process does nothing). One possible contribution is given by the following locus
\begin{center}
\begin{tikzpicture}[scale=1]
%edge
\draw (0,0) to (2,0);
\draw [color=blue] (1,0) node[below]{\tiny$\alpha_1$};

%C_1
\draw [fill=white] (0,0) circle[radius=3pt];
\draw (0,-0.1) node[below]{\small$d$};

%y
\draw [->] (2,0) -- (2,1);
\draw (2,0.9) node[left]{$y$};
\draw [color=blue] (2,0.5) node[right]{\tiny$0$};

%x_1
\draw [->] (2,0) -- (3,0);
\draw (3,0) node[right]{$x_1$};
\draw [color=blue] (2.5,0) node[below]{\tiny$\alpha_1$};

%C_0
\draw [fill=black] (2,0) circle[radius=3pt];
\draw (2,-0.1) node[below]{\small$0$};

%down arrow
\draw [color=blue,->] (1,-0.5) -- (1,-0.9);

%target
\draw [color=blue] (0,-1) to (2,-1);
\draw [fill=blue,color=blue] (0,-1) circle[radius=3pt];
\draw [fill=blue,color=blue] (2,-1) circle[radius=3pt];
\draw [color=blue,->] (2,-1) -- (3,-1);
\draw [color=blue,->] (0,-1) -- (-1,-1);
\end{tikzpicture}
\end{center}
with all other markings located on the genus one vertex. This contributes a nonzero multiple of the invariant $I$ we are trying to compute.

For the other possibilities, first note that, unless every component at the top level contains a single positive-tangency marking and a single node (together with possibly some tangency-zero markings), then $C^\prime$ has fewer than $n$ relative markings and hence the contribution is known recursively. On the other hand, if any non-relative marking other than $y$ is mapped to the top level, $C^\prime$ has $\leq n$ relative markings and less than $m$ non-relative markings, and hence again the contribution is known recursively. The only remaining possibilities are when $x_1$ is replaced by another positive-tangency marking $x_k$ in the above picture; but then the contribution can be calculated by applying the projection formula with $\fgt_{1}$.
\end{enumerate}
We conclude that the only contribution to the right-hand side of \eqref{step 4 recursion 2} which cannot be expressed in terms of previously computed invariants is a nonzero multiple of the invariant we were trying to compute. We thus compute the reduced rubber invariant of $\PP$ with degree $d$, $n$ relative markings and $m$ non-relative markings, which completes Step 4 and the recursion.

\footnotesize

%\bibliographystyle{alpha}
%\bibliography{Bibliography}

\bigskip\bigskip

\noindent Luca Battistella\\
Mathematisches Institut, Ruprecht-Karls-Universit\"at Heidelberg, Im Neuenheimer Feld 205, 69120 Heidelberg, Germany \\
\href{mailto:lbattistella@mathi.uni-heidelberg.de}{lbattistella@mathi.uni-heidelberg.de}\\

\noindent Navid Nabijou \\
Department of Pure Mathematics and Mathematical Statistics, University of Cambridge, Centre for Mathematical Sciences, Wilberforce Road, Cambridge CB3 0WB, United Kingdom \\
\href{mailto:nn333@cam.ac.uk}{nn333@cam.ac.uk}\\

\noindent Dhruv Ranganathan \\
Department of Pure Mathematics and Mathematical Statistics, University of Cambridge, Centre for Mathematical Sciences, Wilberforce Road, Cambridge CB3 0WB, United Kingdom \\
\href{mailto:dr508@cam.ac.uk}{dr508@cam.ac.uk}

\end{document}